\documentclass[12pt,oneside,reqno]{amsart}
\usepackage[colorlinks,linkcolor=black,citecolor=blue]{hyperref}
\usepackage{amssymb}
\usepackage[dvips]{graphics}
\usepackage{txfonts}
\usepackage{bbm}
\usepackage{cases}
\usepackage{amsmath}
\usepackage{graphicx}
\usepackage{mathrsfs}
\usepackage{stmaryrd}
\usepackage{amsfonts}
\newcommand{\comment}[1]{}
\usepackage{enumerate,amsmath,amssymb,amsthm}

\numberwithin{equation}{section}

\allowdisplaybreaks 

\newtheorem{theorem}{Theorem}[section]
\newtheorem{lemma}{Lemma}[section]

\newtheorem{proposition}{Proposition}[section]

\newtheorem{corollary}[theorem]{Corollary}
\newtheorem{remark}{Remark}[section]

\def\v{{{\rm v}}}
\def\e{{\mathrm{e}}}
\def\eps{\varepsilon}

\def\p{\partial}

\def\[{{\Big[}}
\def\]{{\Big]}}
\def\<{{\langle}}
\def\>{{\rangle}}
\def\({{\Big(}}
\def\){{\Big)}}

\def\bx{{\mathbf{x}}}

\def\dif{{\mathord{{\rm d}}}}
\def\dis{{\mathord{{\rm \bf d}}}}

\def\no{\nonumber}
\def\={&\!\!=\!\!&}
\def\bt{\begin{theorem}}
\def\et{\end{theorem}}
\def\bl{\begin{lemma}}
\def\el{\end{lemma}}
\def\br{\begin{remark}}
\def\er{\end{remark}}

\def\bd{\begin{definition}}
\def\ed{\end{definition}}
\def\bp{\begin{proposition}}
\def\ep{\end{proposition}}
\def\bc{\begin{corollary}}
\def\ec{\end{corollary}}
\def\bx{\begin{Examples}}
\def\ex{\end{Examples}}

\def\cF{{\mathcal F}}

\def\cJ{{\mathcal J}}

\def\cL{{\mathcal L}}
\def\cM{{\mathcal M}}

\def\cT{{\mathcal T}}

\def\mC{{\mathbb C}}
\def\mD{{\mathbb D}}
\def\mE{{\mathbb E}}

\def\mH{{\mathbb H}}
\def\mI{{\mathbb I}}

\def\mL{{\mathbb L}}

\def\mN{{\mathbb N}}

\def\mP{{\mathbb P}}

\def\mR{{\mathbb R}}
\def\mS{{\mathbb S}}

\def\mV{{\mathbb V}}
\def\mW{{\mathbb W}}

\def\bP{{\mathbf P}}

\def\sA{{\mathscr A}}

\def\sC{{\mathscr C}}

\def\sF{{\mathscr F}}

\def\sH{{\mathscr H}}
\def\sI{{\mathscr I}}

\def\sK{{\mathscr K}}
\def\sL{{\mathscr L}}

\def\sS{{\mathscr S}}

\def\sU{{\mathscr U}}
\def\sV{{\mathscr V}}
\def\sW{{\mathscr W}}

\def\geq{\geqslant}
\def\leq{\leqslant}

\def\div{\mathord{{\rm div}}}

\def\bP{{\mathbf P}}
\def\bE{{\mathbf E}}

\newtheorem{definition}{Definition}

\def\mN{{\mathbb N}}
\def\mR{{\mathbb R}}

\numberwithin{equation}{section}

\begin{document}

\title{H\"ormander's hypoelliptic theorem for nonlocal operators}

\author{Zimo Hao, Xuhui Peng and Xicheng Zhang}

\address{Zimo Hao:
School of Mathematics and Statistics, Wuhan University,
Wuhan, Hubei 430072, P.R.China\\
Email: zimohao@whu.edu.cn
 }

 \address{Xuhui Peng:
School of Mathematics and Statistics, Hunan Normal University,
Changsha, Hunan, P.R.China\\
Email: xhpeng@hunnu.edu.cn 
 }

\address{Xicheng Zhang:
School of Mathematics and Statistics, Wuhan University,
Wuhan, Hubei 430072, P.R.China\\
Email: XichengZhang@gmail.com
 }

\begin{abstract}
In this paper we show the H\"ormander hypoelliptic theorem for nonlocal operators by a purely probabilistic method:  the Malliavin calculus.
Roughly speaking, under general H\"ormander's Lie bracket conditions, we show the regularization effect of discontinuous L\'evy noises
for possibly degenerate stochastic differential equations with jumps. To treat the large jumps, we use the perturbation argument together with 
interpolation techniques and some short time asymptotic estimates of the semigroup. As an application, 
we show the existence of fundamental solutions for operator $\p_t-\sK$, where $\sK$ is 
the nonlocal kinetic operator:
$$
\sK  f(x,{\rm v}):={\rm p.v}\int_{\mR^d}(f(x,{\rm v}+w)-f(x,{\rm v}))\frac{\kappa(x,{\rm v},w)}{|w|^{d+\alpha}}\dif w
+{\rm v}\cdot\nabla_x f(x,{\rm v})+b(x,{\rm v})\cdot\nabla_{\rm v} f(x,{\rm v}).
$$
Here $\kappa_0^{-1}\leq \kappa(x,{\rm v},w)\leq\kappa_0$ belongs to $C^\infty_b(\mR^{3d})$ and is symmetric in $w$, p.v. stands for the Cauchy principal value,
and $b\in C^\infty_b(\mR^{2d};\mR^d)$.


 
 \end{abstract}

\thanks{
X.P. is partially supported by NSFC (No.11501195) and a Scientific Research Fund of Hunan
Provincial Education Department (No.17C0953).
X.Z. is partially supported by NNSFC grant of China (No. 11731009) and the DFG through the CRC 1283 
``Taming uncertainty and profiting from randomness and low regularity in analysis, stochastics and their applications''.  }

\maketitle \rm

\section{Introduction}

\noindent{\bf 1.1 Introduction.} Let $\sA$ be a differential operator in $\mR^d$ with smooth coefficients. Hypoellipticity in the theory of PDEs means that 
for any distribution $u$ and open subset $U\subset\mR^d$, 
$$
\sA u|_{U}\in C^\infty(U)\Rightarrow u\in C^\infty(U).
$$
Let $A_0,A_1,\cdots A_d$ be $d+1$-differential operators of first order (or vector fields) with smooth coefficients and $c$ a smooth function.
The classical H\"ormander's hypoelliptic theorem tells us that if 
$$
\mbox{$\sV:=$\rm Lie($A_1,\cdots,A_d, [A_0,A_1],\cdots,[A_0,A_d]$)=$\mR^d$},
$$
where $[A_0,A_k]:=A_0A_k-A_kA_0$ is the Lie bracket, and $\sV$ stands for the Lie algebra generated by 
vector fields $A_k, [A_0,A_k], k=1,\cdots,d$,
then $\sA:=\sum_{k=1}^dA_k^2+A_0+c-\frac{\p}{\p t}$ is hypoelliptic in $\mR^{d+1}$ (cf. \cite{Ho}, \cite{Wi}, \cite{Ha} and \cite{Kr}, etc.).

\medskip

Consider the following It\^o's type SDE
\begin{align}\label{SDE-1}
\dif X_t=b(X_t)\dif t+\sigma_k(X_t)\dif W^k_t,\ X_0=x\in\mR^d,
\end{align}
where $W$ is a $d$-dimensional standard Brownian motion and $b,\sigma_k:\mR^d\to\mR^d$, $k=1,\cdots, d$ are $C^\infty_b$-functions.
Here and below we use Einstein's summation convention: If an index appears twice in a product, then it will be summed automatically.
We now define $d+1$-vector fields by
\begin{align}\label{VEC}
A_0:=\left(b^i-\tfrac{1}{2}\sigma^j_k\partial_j\sigma_k^i\right)\p_{i},\ \ A_k:=\sigma^i_k\p_{i},\ k=1,\cdots,d,
\end{align}
where $\p_i:=\p_{x_i}=\frac{\p}{\p x_i}$.
Let $\mu_t(x,\dif y)$ be the distributional density of the unique solution $X_t(x)$ of SDE \eqref{SDE-1}.
By It\^o's formula, one sees that in the distributional sense,
$$
\p_t \mu_t(x,\cdot)=\left(\tfrac{1}{2}A^2_k+A_0\right)^*\mu_t(x,\cdot),
$$
where the asterisk stands for the adjoint operator. Notice that
$$
\left(\tfrac{1}{2}A^2_k+A_0\right)^*=\tfrac{1}{2}A^2_k+\widehat A_0+c,
$$
where $\widehat A_0:=(\sigma^i_k\p_{j}\sigma^j_k)\p_{i}-A_0$ and $c:=\p_i\p_j(\sigma^i_k\sigma^j_k)/2-\div b$.
By H\"ormander's hypoelliptic theorem, if
$$
\mbox{$\sV=$\rm Lie($A_1,\cdots,A_d, [\widehat A_0,A_1],\cdots,[\widehat A_0,A_d]$)=$\mR^d$},
$$
then $\mu_t(x,\cdot)$ admits a smooth density (see \cite{Wi}). In \cite{Ma}, Malliavin provides a purely probabilistic proof for the above result by
infinitely dimensional stochastic calculus of variations invented by him, which is now called the Malliavin calculus (see \cite{Nu}). 
Since then, the Malliavin calculus has been developed very well, and emerged in many fields such as financial, control, filtering, and so on. 
Notice that in \cite{Ha}, Hairer presents a short and self-contained proof for  H\"ormander's theorem based on Malliavin's idea.

\medskip

In this paper we are concerned with the following SDE with jumps:
\begin{align}\label{SDE}
\dif X_t=b(X_t)\dif t+\sigma_k(X_t)\dif W^k_t+\int_{\mR^d_0}g(X_{t-},z)\widetilde{N}(\dif t,\dif z),\ X_0=x\in\mR^d,
\end{align}
where $\mR^d_0:=\mR^d\setminus\{0\}$, and $N(\dif t,\dif z)$ is a Poisson random measure with intensity $\dif t\nu(\dif z)$,
$$
\widetilde N(\dif t,\dif z):=N(\dif t,\dif z)-\dif t\nu(\dif z),
$$
and $\nu$ is a symmetric L\'evy measure over $\mR^d_0$, and $b,\sigma_k:\mR^d\to\mR^d, k=1,\cdots,d$ and $g:\mR^d\times\mR^d_0\to\mR^d$
are smooth Lipschitz functions. It is well known that
SDE \eqref{SDE} admits a unique strong solution $X_t(x)$ for each initial value $x\in\mR^d$ (for example, see \cite{Pr}).
Suppose $g(x,-z)=-g(x,z)$. By It\^o's formula, one sees that the generator of SDE \eqref{SDE} is given by
\begin{align}\label{LL}
\sA \varphi(x):=\frac{1}{2}A^2_k\varphi(x)+A_0\varphi(x)+{\rm p.v.}\int_{\mR^d_0}\Big(\varphi(x+g(x,z))-\varphi(x)\Big)\nu(\dif z),
\end{align}
where p.v. stands for the Cauchy principal value, and $A_0,A_k$ are defined by \eqref{VEC}.
More precisely, for $\varphi\in C^\infty_b(\mR^d)$, if we define
\begin{align}\label{KP8}
\cT_t\varphi(x):=\mE\varphi(X_t(x)),
\end{align}
then
\begin{align}\label{KP0}
\p_t\cT_t\varphi=\sA\cT_t\varphi=\cT_t\sA \varphi.
\end{align}
The aim of this work is to show that under full H\"ormander's conditions, the solution $X_t(x)$ of SDE \eqref{SDE} admits a smooth density.

\medskip

The smoothness of  the distribution density of the solutions to SDEs with jumps has been studied extensively since Malliavin's initiated work.
In \cite{Bi0}, Bismut put forward a simple argument: Girsanov's transformation to study the smoothness of the distribution densities to SDEs with jumps.
In \cite{Bi-Gr-Ja}, Bichteler, Gravereaux and Jacod give a systematic introducion for the Malliavin calculus with jumps.
In \cite{Pi}, Picard used the difference operators to present another criterion for the smoothness of the distribution densities of Poisson functionals,
see \cite{Is-Ku} for recent development for Wiener-Poisson functionals. Under partial H\"ormander's conditions, there
are also several works to study the smoothnees of degenerate SDEs with jumps.
In \cite{Ca}, Cass established a H\"ormander's type theorem for SDEs with jumps
by proving a Norris' type lemma for discontinuous semimartingales. However, the Brownian noise can not disappear.
In the pure jump degenerate case, Komatsu and Takeuchi \cite[Theorem 3]{KT2001} introduced a quite useful estimate for discontinuous semimartingales, and then proved
a H\"ormander's type theorem for SDEs with jumps. Some subsequent results based on Komatsu-Takeuchi's estimate are referred to \cite{Ta,Kuni}.
Unfortunately, there is a gap in the proof of \cite[Theorem 3]{KT2001}. 
We fill it up in \cite{Zhang16} in a slightly different form (see Lemma \ref{Le42} below).
Basing on this new form of Komatsu-Takeuchi's type estimate, we prove a H\"ormander's type theorem for pure jump SDEs with {\it nonzero drifts} in \cite{Zhang17}. Other works about the regularization of jump noises can be found in \cite{Ba-Cl, No-Si, Is-Ku-Ts} and references therein.

\medskip

In Malliavin's probabilistic proof of H\"ormander's hypoelliptic theorem, one of the key steps is to show the 
$L^p$-integrability of the inverse of the Malliavin covariance matrix. In the nondegenerate full noise case, it is relatively easy to obtain.
However, under H\"ormander's Lie bracket conditions, it is a quite challenge problem. In particular, Norris \cite{No} provides an important estimate 
for general continuous semimartingales to treat this (see \cite{Nu}). Now it is usually called Norris' lemma (see \cite[Lemma 4.11]{Ha} for an elegant proof), 
which can be considered as a quantitative version of Doob-Meyer's decomposition theorem. For general discontinuous semimartingales, 
Komatsu-Takeuchi's estimate should be regarded as a substitution of Norris' lemma.
We shall use it to prove a full H\"ormander's theorem for SDEs with jumps, see Theorem \ref{Th1} below.

\medskip

One of the motivations of studying nonlocal H\"ormander's hypoelliptic theorem comes from the study of spatial inhomeogenous Boltzmann's equations.
It is well known that the linearized spatial inhomeogenous Boltzmann's equation can be written as the following form
that involves non-local operator of fractional Laplacian type (cf. \cite{Vi} and \cite{Ch-Zh1}):
\begin{align}\label{HA3}
\p_t f+\v\cdot \nabla_x f={\rm p.v.} \int_{\mR^d}(f(\cdot+w)-f(\cdot))\frac{K_g(\cdot,w)}{|w|^{\alpha+d}}\dif w+f\,H_g,
\end{align}
where $f$ and $g$ are functions of $x,\v$ and $w$, and
$$
K_g(\v,w):=2\int_{\{h\cdot w=0\}}g(\v-h)|h-w|^{\gamma+1+\alpha}\dif h,
$$
and
$$
H_g(\v):=2\int_{\mR^d}\!\!\int_{\{h\cdot w=0\}}(g(\v-h)-g(\v-h+w))\frac{|h-w|^{\gamma+1+\alpha}}{|w|^{\alpha+d}}\dif h\dif w.
$$
Here $\gamma+\alpha\in(-1,1)$. Note that $K_g$ is a symmetric kernel in $w$, i.e.,
$K_g(\cdot,w)=K_g(\cdot,-w)$, and $f H_g$ is a zero order term in $f$. We shall see in Section 7 that the principal part of \eqref{HA3} 
can be written as the form of \eqref{LL}.

\medskip

\noindent{\bf 1.2 Main results.} 
To make our statement of main results as  simple and apparent as possible, throughout this paper we assume that for some $\alpha\in(0,2)$,
$$
\nu(\dif z)=\dif z/|z|^{d+\alpha}.
$$
We also introduce the following assumptions about $b,\sigma_k$ and $g$: for some $\ell\in\mN\cup\{\infty\}$,
\begin{enumerate}[{\bf (H$_{\ell}$)}]
\item For any $i\in\mN$ and $j=0,\cdots,\ell$, there are $C_i, C_{ij}\geq 1$ such that for all $x\in\mR^d$ and $|z|<1$,
$$     
|\nabla^{i} b(x)|+|\nabla^{i}\sigma_k(x)|\leq C_i,\ \ |\nabla_x^{i}\nabla_z^jg(x,z)|\leq C_{ij}|z|^{1-j}.
$$
Moreover, we require $g(x,-z)=-g(x,z)$ and for some $\beta\in(0,1]$,
$$
|\nabla_zg(x,z)-\nabla_zg(x,0)|\leq C_x |z|^\beta,\ \ |z|<1,
$$
where $C_x>0$ continuously depends on $x\in\mR^d$.
\end{enumerate}
\begin{enumerate}[{\bf (H$_g^{\rm o}$)}]
   \item It holds that $\inf_{x,z\in \mR^d}\det (\mI+\nabla_xg(x,z))>0$ and supp$\{g(x,\cdot)\}\subset B_1$.
\end{enumerate}
\br
It should be kept in mind that $g(x,z)=\widetilde\sigma(x)z$ with $\widetilde\sigma:\mR^d\to\mR^d\otimes\mR^d$ 
satisfying $\|\nabla^i\widetilde\sigma\|_\infty\leq C_i$ for $i\in\mN$, fullfills the assumptions about $g$ in {\bf (H$_\ell$)}.
Moreover, in order to make {\bf (H$_g^{\rm o}$)} hold, one needs to assume $g(x,z)=\widetilde\sigma(x)z\cdot 1_{|z|\leq\delta}$
with $\delta$ being small enough so that SDE \eqref{SDE} defines a stochastic diffeomorphism flows in $\mR^d$.
\er

Let $A_0,A_k$ be as in \eqref{VEC} and $\widetilde A_k(x):=\p_{z_k}g^i(x,0)\p_i$. Define
$$
\sV_0:=\{A_k,\widetilde A_k, k=1,\cdots,d\},
$$ 
and for $j=1,2,\cdots,$
 \begin{align}
   \sV_j:=\Big\{[A_k,V], [\widetilde A_k,V],[A_0,V]:   V\in\sV_{j-1},k=1,\cdots,d\Big\}.\label{VJ}
 \end{align}
The following strong H\"ormander's condition is imposed:
\begin{enumerate}[{\bf (H$^{\rm str}_{\bf or}$)}]
   \item For some $j_0\in\{0\}\cup\mN$, span$\{\cup_{j=0}^{j_0}\sV_j\}=\mR^d$ at each point $x\in\mR^d$.
    \end{enumerate}

We aim to prove the following result.
\bt\label{Th1}
Under {\bf (H$_{\infty}$)$+$(H$_g^{\rm o}$)$+$(H$^{\rm str}_{\bf or}$)},
there is a nonnegative smooth function $\rho_t(x,y)$ on $(0,\infty)\times\mR^d\times\mR^d$ so that
$$
\mP\circ X^{-1}_t(x)(\dif y)=\rho_t(x,y)\dif y,
$$ 
where $X_t(x)$ is the solution of SDE \eqref{SDE} with starting point $X_0(x)=x$,
and
\begin{align}
\p_t\rho_t(x,y)=\sA\rho_t(\cdot,y)(x)=\sA^*\rho_t(x,\cdot)(y),\ \ \lim_{t\downarrow 0}\rho_t(x,y)=\delta_x(\dif y),
\end{align}
where $\sA^*$ is the adjoint operator of $\sA$ (see \eqref{LL}), and $\delta_x$ is the Dirac measure concentrated at $x$.
\et

To treat the large jumps, we make the following stronger assumptions:
\begin{enumerate}[{\bf (H$'_\ell$)}]
   \item In addition to {\bf (H$_\ell$)}, we assume that $\cup_{j=0}^\infty\sV_j\subset C_b^\infty(\mR^d)$ and
   $$
   |\nabla_z^jg(x,z)|\leq C_{j}|z|^{1-j},\ |z|<1, \ j\in\mN_0.
   $$
    \end{enumerate}
\begin{enumerate}[{\bf (H$^{\rm uni}_{\bf or}$)}]
   \item The following uniform H\"ormander's condition holds: for some $j_0\in \mN_0$ and $c_0>0$,
\begin{align}\label{Hor}
   \inf_{x\in\mR^d}\inf_{|u|=1}\sum_{j=0}^{j_0}\sum_{V \in \sV_j}|uV(x)|^2\geq c_0.
\end{align}
    \end{enumerate}
\br
In {\bf (H$'_\ell$)}, the drift $b$ may be linear growth, but $\sigma$ and $g$ are bounded in $x$. If $b$ is also bounded, 
then for each $V\in\cup_{j=0}^\infty\sV_j$,
it automatically holds that $V\in C^\infty_b(\mR^d)$.
\er

By a perturbation argument, we can prove the following result. Since its proof is completely the same as in \cite[Theorem 1.2]{Zhang16}, we omit the details.
\bt\label{Main}
Let $\sL$ be a bounded linear operator in Sobolev space $\mW^{k,p}(\mR^d)$ for any $p>1$ and $k\in\mN_0$. 
Under {\bf (H$'_2$)}$+${\bf (H$^{\rm o}_g$)}$+${\bf (H$^{\rm uni}_{\bf or}$)},  
there exists a continuous function $\rho_t(x,y)$ 
on $(0,\infty)\times\mR^d\times\mR^d$ called fundamental solution of operator $\sA+\sL$ with the properties that
\begin{enumerate}[(i)]
\item For each $t>0$ and $y\in\mR^d$, the mapping $x\mapsto\rho_t(x,y)$ is smooth, and there is a $\gamma=\gamma(\alpha,j_0,d)>0$ such that 
for any $p\in(1,\infty)$, $T>0$ and $k\in\mN_0$,
\begin{align}
\|\nabla^k_x\rho_t(x,\cdot)\|_p\leq C t^{-(k+d)\gamma}, \ \ \forall (t,x)\in(0,T]\times\mR^d.\label{NG2}
\end{align}
\item For any $p\in(1,\infty)$ and $\varphi\in L^p(\mR^d)$, 
$\cT_t\varphi(x):=\int_{\mR^d}\varphi(y)\rho_t(x,y)\dif y\in \cap_{k}\mW^{k,p}(\mR^d)$ satisfies
\begin{align}
\p_t \cT_t\varphi(x)=(\sA+\sL)\cT_t\varphi(x), \ \ \forall (t,x)\in(0,\infty)\times\mR^d.\label{NG1}
\end{align}
\end{enumerate} 
\et

The above result provides a way of treating the large jumps. In applications, we usually take $\sL$ as the large jump operator, for example,
$$
\sL\varphi(x):= \int_{|z|\geq \delta}\Big(\varphi(x+z)-\varphi(x)\Big)\kappa(x,z)\nu(\dif z),\ \delta>0.
$$
In fact, we shall apply Theorem \ref{Main} to the nonlocal kinetic operators in Section 7. 
However, sometimes it is not easy to verify the boundedness of the large jump operator in $\mW^{k,p}$.
The following theorem provides part results for general SDE \eqref{SDE} without assuming {\bf (H$^{\rm o}_g$)},
which is still based on the perturbation argument and suitable interpolation techniques as in \cite{Zhang16}. 
\bt\label{Th2}
Under {\bf (H$'_2$)}$+${\bf (H$^{\rm uni}_{\bf or}$)} and $g\in C^\infty_b(\mR^d\times B_1^c)$, where $B_1$ is the unit ball,
there is a nonnegative measurable function $\rho_t(x,y)$ 
on $(0,\infty)\times\mR^d\times\mR^d$ so that 
\begin{enumerate}[(i)]
\item For each $t>0$ and $x\in\mR^d$, $\mP\circ X^{-1}_t(x)(\dif y)=\rho_t(x,y)\dif y$, where $X_t(x)$ is the solution of SDE \eqref{SDE} with starting point $X_0(x)=x$.
\item There are $\eps_0, \vartheta_0,q_0$ and $\gamma$ such that for all $\eps\in[0,\eps_0)$, $\vartheta\in[0,\vartheta_0)$ and $q\in[1,q_0)$,
$$
\sup_{x\in\mR^d}\|(\mI-\Delta)^{\frac{\alpha+\eps}{2}}_x\Delta^{\frac{\vartheta}{2}}_y\rho_t(x,\cdot)\|_q
\leq Ct^{-\gamma},\ t\in(0,1).
$$
\item If the support of $g(x,\cdot)$ is contained in a ball $B_R$ for all $x\in\mR^d$, where $R\geq 1$, then for any $k\in\mN_0$,
there are $\vartheta_0,q_0$ and $\gamma_k$ such that for all $\vartheta\in[0,\vartheta_0)$ and $q\in[1,q_0)$,
$$
\sup_{x\in\mR^d}\|\nabla^k_x\Delta^{\frac{\vartheta}{2}}_y\rho_t(x,\cdot)\|_q
\leq Ct^{-\gamma_k},\ t\in(0,1).
$$
\end{enumerate}
\et
\br
The above (ii) implies that for any
$\varphi\in L^\infty(\mR^d)$, 
$$
\cT_t\varphi(x):=\int_{\mR^d}\varphi(y)\rho_t(x,y)\dif y\in \mC^{\alpha+\eps},
$$ 
where $\mC^{\alpha+\eps}$ is the usual H\"older space.
In particular, the strong Feller property holds for $\cT_t$.
Moreover, if $\sigma_k\equiv 0$ and $\alpha\in[1,2)$, then $\cT_t\varphi$ satisfies the following nonlocal equation in the classical sense
$$
\p_t\cT_t\varphi=\sA\cT_t\varphi,\ t>0.
$$ 
\er

\noindent{\bf 1.3 Examples.} Below we provide several simple examples to illustrate our results.

\medskip

\noindent {\bf Example 1.} (A standard nonlinear example) Let $L_t$ be an one dimensional L\'evy process with L\'evy measure $\nu(\dif z)=\dif z/|z|^{1+\alpha}$, where $\alpha\in(0,2)$.
Consider the following SDE:
$$
\dif X_t=-\sin(X_t)\dif t+\cos(X_t)\dif L_t,\ \ X_0=x.
$$
In this case, $A_0=-\sin(x)\p_x$ and $\widetilde A_1=\cos(x)\p_x$.
The generator of $X_t$ is given by
$$
\sA \varphi(x):=-\sin(x)\varphi'(x)+{\rm p.v.}\int_{\mR}\Big(\varphi(x+\cos(x)z)-\varphi(x)\Big)\nu(\dif z).
$$
Clearly, $[A_0,\widetilde A_1]=\p_x$ and {\bf (H$^{\rm uni}_{\bf or}$)} holds with $c_0=1$ in \eqref{Hor}.

\medskip
\noindent {\bf Example 2.} (Nonlocal Grushin's type operator) Let $L_t=(L^1_t,L^2_t)$ be a two-dimensional 
L\'evy process with L\'evy measure $\nu(\dif z)=1_{|z|\leq 1}|z|^{-2-\alpha}\dif z$, where
$\alpha\in(0,2)$. Let $X_t=(X^1_t,X^2_t)$ solve the following SDE:
$$
\left\{
\begin{aligned}
&\dif X^1_t=\dif L^1_t,& X^1_0=x_1,\\
&\dif X^2_t=X^1_t\dif L^2_t, & X^2_0=x_2.
\end{aligned}
\right.
$$
In this case, $A_k=0$ for $k=0,1,2$, $\widetilde A_1=\p_{x_1}$, $\widetilde A_2=x_1\p_{x_2}$, and the generator of $X_t$ is given by
$$
\sA \varphi(x):={\rm p.v.}\int_{\mR^2}\Big(\varphi(x_1+z_1,x_2+x_1z_2)-\varphi(x)\Big)\nu(\dif z).
$$
Clearly, $[\widetilde A_1,\widetilde A_2]=\p_{x_2}$
and {\bf (H$^{\rm uni}_{\bf or}$)} holds with $c_0=2$ in \eqref{Hor}.

\medskip

\noindent {\bf Example 3.} (Local and nonlocal Grushin's type operator) 
Let $L_t$ be an one-dimensional L\'evy process with L\'evy measure $\nu(\dif z)=|z|^{-1-\alpha}\dif z$, where
$\alpha\in(0,2)$ and $W_t$ an one-dimensional Brownian motion. Let $X_t=(X^1_t,X^2_t)$ solve the following SDE:
$$
\left\{
\begin{aligned}
&\dif X^1_t=\dif L_t,& X^1_0=x_1,\\
&\dif X^2_t=X^1_t\dif W_t,& X^2_0=x_2.
\end{aligned}
\right.
$$
In this case, $A_0=A_1=0, A_2=x_1\p_{x_2}$, $\widetilde A_1=\p_{x_1}$, $\widetilde A_2=0$, and the generator of $X_t$ is given by
$$
\sA \varphi(x):=\frac{1}{2}x^2_1\p^2_{x_2}\varphi(x)+{\rm p.v.}\int_{\mR}\Big(\varphi(x_1+z_1,x_2)-\varphi(x)\Big)\nu(\dif z_1).
$$
Clearly, $[\widetilde A_1,A_2]=\p_{x_2}$ and {\bf (H$^{\rm uni}_{\bf or}$)} holds with $c_0=2$ in \eqref{Hor}.

\medskip

\noindent {\bf Example 4.} (Nonlocal Kolmogorov's type operator)  Let $L_t$ be an one-dimensional L\'evy process with L\'evy measure $\nu(\dif z)=|z|^{-1-\alpha}\dif z$, where
$\alpha\in(0,2)$. Let $X_t=(X^1_t,X^2_t)$ solve the following SDE:
$$
\left\{
\begin{aligned}
&\dif X^1_t=X^2_t\dif t,& X^1_0=x_1,\\
&\dif X^2_t=\dif L_t-X^1_t\dif t,& X^2_0=x_2.
\end{aligned}
\right.
$$
In this case, $A_1=A_2=0, A_0=x_2\p_{x_1}-x_1\p_{x_2}$, $\widetilde A_1=0$, $\widetilde A_2=\p_{x_2}$, and the generator of $X_t$ is given by
$$
\sA \varphi(x):=x_2\p_{x_1}\varphi(x)-x_1\p_{x_2}\varphi(x)+{\rm p.v.}\int_{\mR}\Big(\varphi(x_1,x_2+z_2)-\varphi(x)\Big)\nu(\dif z_2).
$$
Clearly, $[A_0,\widetilde A_2]=\p_{x_1}$ and {\bf (H$^{\rm uni}_{\bf or}$)} holds with $c_0=2$ in \eqref{Hor}.

\medskip

\noindent {\bf Example 5.} (Nonlocal relativistic operator)  Let $L_t$ be an one-dimensional L\'evy process with 
L\'evy measure $\nu(\dif z)=|z|^{-1-\alpha}\dif z$, where
$\alpha\in(0,2)$. Let $Z_t=(X_t,V_t)$ solve the following SDE:
$$
\left\{
\begin{aligned}
&\dif X_t=V_t/{\sqrt{1+|V_t|^2}}\dif t,& X_0=x,\\
&\dif V_t=\dif L_t-X_t\dif t,& V_0={\rm v}.
\end{aligned}
\right.
$$
In this case, $A_1=A_2=0, A_0={\rm v}/\sqrt{1+|{\rm v}|^2}\p_{x}-x\p_{{\rm v}}$, $\widetilde A_1=0$, $\widetilde A_2=\p_{{\rm v}}$, and the generator of $Z_t$ is given by
$$
\sA \varphi(x,{\rm v}):=\frac{{\rm v}}{\sqrt{1+|{\rm v}|^2}}\p_{x}\varphi(x,{\rm v})-x\p_{\rm v}\varphi(x,{\rm v})+{\rm p.v.}\int_{\mR}\Big(\varphi(x,{\rm v+v'})-\varphi(x,{\rm v})\Big)\nu(\dif w).
$$
Clearly, $[A_0,\widetilde A_2]=(1+|{\rm v}|^2)^{-3/2}\p_{x}$ and {\bf (H$^{\rm str}_{\bf or}$)} holds.

\medskip

\noindent{\bf 1.4 Structure.} This paper is organized as follows: In Section 2 we recall Bismut's approach to the Malliavin calculus of Wiener-Poisson functionals.
In Section 3, we recall and prove an improved Komatsu-Takeuchi's type estimate. In Section 4, we show the key estimate of 
the Laplace transform of the reduced Malliavin matrix. In Section 5, we prove Theorem \ref{Th1}. In Section 6 we prove Theorem \ref{Th2}.
Finally, in Section 7, we apply our main result to the nonlocal kinetic operators and show the existence of smooth fundamental solutions, where 
the key point is to write the nonlocal operator as the generator of an SDE. For this aim, we need to solve a relaxed Jacobi equation.
\section{Preliminaries}

In this subsection, we recall some basic facts about Bismut's approach to the Malliavin calculus with jumps (see \cite[Section 2]{So-Zh}).
Let $\Gamma\subset\mR^d$ be an open set containing the origin. We define
\begin{align}
\Gamma_0:=\Gamma\setminus\{0\},\ \ \varrho(z):=1\vee\dis(z,\Gamma^c_0)^{-1},\label{Rho1}
\end{align}
where $\dis(z,\Gamma^c_0)$ is the distance of $z$ to the complement of $\Gamma_0$. Notice that $\varrho(z)=\frac{1}{|z|}$ near $0$.

Let $\Omega$ be the canonical space of all points $\omega=(w,\mu)$, where 
\begin{itemize}
\item $w: [0,1]\to\mR^d$ is a continuous function with $w(0)=0$;
\item $\mu$ is an integer-valued measure on $[0,1]\times\Gamma_0$ with $\mu(A)<+\infty$ for any compact set $A\subset[0,1]\times\Gamma_0$.
\end{itemize} 
Define the canonical process on $\Omega$ as follows: for $\omega=(w,\mu)$,
$$
W_t(\omega):=w(t),\ \ \ N(\omega; \dif t,\dif z):=\mu(\omega; \dif t,\dif z):=\mu(\dif t,\dif z).
$$
Let $(\sF_t)_{t\in[0,1]}$ be the smallest right-continuous filtration on $\Omega$ such that $W$ and $N$ are optional. 
In the following, we write $\sF:=\sF_1$, and endow $(\Omega,\sF)$ with the unique probability measure $\mP$ such that 
\begin{itemize}
\item $W$ is a standard $d$-dimensional Brownian motion;
\item $N$ is a Poisson random measure with intensity $\dif t\nu(\dif z)$, where $\nu(\dif z)=\kappa(z)\dif z$ with
\begin{align}
\kappa\in C^1(\Gamma_0;(0,\infty)),\  \int_{\Gamma_0}(1\wedge|z|^2)\kappa(z)\dif z<+\infty,\ \ |\nabla\log\kappa(z)|\leq C\varrho(z),\label{ET1}
\end{align}
where $\varrho(z)$ is defined by (\ref{Rho1}). In the following we write
$$
\widetilde N(\dif t,\dif z):=N(\dif t,\dif z)-\dif t\nu(\dif z).
$$
\end{itemize}
Let $p\geq 1$ and $m\in\mN$. We introduce the following spaces for later use.
\begin{itemize}
\item $\mL^1_p$: The space of all predictable processes: $\xi:\Omega\times[0,1]\times\Gamma_0\to\mR^m$ with finite norm:
$$
\|\xi\|_{\mL^1_p}:=\left[\mE\left(\int^1_0\!\!\!\int_{\Gamma_0}|\xi(s,z)|\nu(\dif z)\dif s\right)^p\right]^{\frac{1}{p}}
+\left[\mE\int^1_0\!\!\!\int_{\Gamma_0}|\xi(s,z)|^p\nu(\dif z)\dif s\right]^{\frac{1}{p}}<\infty.
$$
\item $\mL^2_p$: The space of all predictable processes: $\xi:\Omega\times[0,1]\times\Gamma_0\to\mR^m$ with finite norm:
$$
\|\xi\|_{\mL^2_p}:=\left[\mE\left(\int^1_0\!\!\!\int_{\Gamma_0}|\xi(s,z)|^2\nu(\dif z)\dif s\right)^{\frac{p}{2}}\right]^{\frac{1}{p}}
+\left[\mE\int^1_0\!\!\!\int_{\Gamma_0}|\xi(s,z)|^p\nu(\dif z)\dif s\right]^{\frac{1}{p}}<\infty.
$$
\item $\mH_p$: The space of all measurable adapted processes $h:\Omega\times[0,1]\to\mR^d$ with finite norm:
$$
\|h\|_{\mH_p}:=\left[\mE\left(\int^1_0|h(s)|^2\dif s\right)^{\frac{p}{2}}\right]^{\frac{1}{p}}<+\infty.
$$
\item $\mV_p$:  The space of all predictable processes $\v: \Omega\times[0,1]\times\Gamma_0\to\mR^d$ with finite norm:
$$
\|\v\|_{\mV_p}:=\|\nabla\v\|_{\mL^1_p}+\|\v\varrho\|_{\mL^1_p}<\infty,
$$
where $\varrho(z)$ is defined by (\ref{Rho}). Below we shall write
$$
\mH_{\infty-}:=\cap_{p\geq 1}\mH_p,\ \ \mV_{\infty-}:=\cap_{p\geq 1}\mV_p.
$$
\item $\mH_0$:  The space of all bounded measurable adapted  processes $h:\Omega\times[0,1]\to\mR^d$.
\item $\mV_0$:  The space of all predictable processes $\v: \Omega\times[0,1]\times\Gamma_0\to\mR^d$ 
with the following properties: (i) $\v$ and $\nabla_z \v$ are bounded;
(ii) there exists a compact subset $U\subset \Gamma_0$ such that
$$
\v(t,z)=0,\ \ \forall z\notin U.
$$

\item For any $p\geq 1$, $\mV_0$ (resp. $\mH_0$) is dense in $\mV_p$ (resp.  $\mH_p$).
\end{itemize}

Let $C_p^\infty(\mR^m)$ be the class of all smooth functions on $\mR^m$ whose derivatives of all orders have at most polynomial growth. 
Let $\cF C^\infty_p$ be the class of all Wiener-Poisson functionals on $\Omega$ with the following form:
$$
F(\omega)=f(w(h_1),\cdots, w(h_{m_1}), \mu(g_1),\cdots, \mu(g_{m_2})),\ \ \omega=(w,\mu)\in\Omega,
$$
where $f\in C_p^\infty(\mR^{m_1+m_2})$, $h_1,\cdots, h_{m_1}\in\mH_0$ and $g_1,\cdots, g_{m_2}\in\mV_0$ are non-random, and
$$
w(h_i):=\int^1_0\<h_i(s), \dif w(s)\>_{\mR^d},\ \ \mu(g_j):=\int^1_0\!\!\!\int_{\Gamma_0}g_j(s,z)\mu(\dif s,\dif z).
$$
Notice that
$$
\cF C^\infty_p\mbox{ is dense in } \cap_{p\geq 1}L^p(\Omega,\sF,\mP).
$$
For $F\in \cF C^\infty_p$ and $\Theta=(h,\v)\in\mH_{\infty-}\times\mV_{\infty-}$, define
\begin{align}
D_\Theta F&:=\sum_{i=1}^{m_1}\p_i f\!\!\int^1_0\<h(s), h_i(s)\>_{\mR^d}\dif s
+\sum_{j=1}^{m_2}\p_{j+m_1} f\!\!\int^1_0\!\!\!\int_{\Gamma_0}\nabla_\v g_j(s,z)\mu(\dif s,\dif z),\label{PT1}
\end{align}
where $\nabla_\v g_j(s,z):=\v_i(s,z)\p_{z_i} g_j(s,z)$.

We have the following integration by parts formula (cf. \cite[Theorem 2.9]{So-Zh}).
\bt\label{Th21}
Let $\Theta=(h,\v)\in\mH_{\infty-}\times\mV_{\infty-}$ and $p>1$.
The linear operator $(D_\Theta, \cF C^\infty_p)$ is closable in $L^p(\Omega)$. The closure is denoted by
$(D_\Theta,\mW^{1,p}_\Theta(\Omega))$, which is a Banach space with respect to the norm:
$$
\|F\|_{\Theta; 1,p}:=\|F\|_{L^p}+\|D_\Theta F\|_{L^p}.
$$
Moreover, we have the following consequences:
\begin{enumerate}[(i)]
\item For any $F\in\mW^{1,p}_\Theta(\Omega)$, the following integration by parts formula holds:
\begin{align}
\mE(D_\Theta F)=\mE(F \div(\Theta)),\label{ER88}
\end{align}
where $\div(\Theta)$ is defined by
\begin{align}
\div\Theta:=\int^1_0\<h(s),\dif W_s\>_{\mR^d}-\int^1_0\!\!\!\int_{\Gamma_0}\frac{\div(\kappa\v)(s,z)}{\kappa(z)}\widetilde N(\dif s,\dif z).\label{ER1}
\end{align}
\item For $m,k\in\mN$ and $F=(F_1,\cdots, F_m)\in (\mW^{1,\infty-}_\Theta)^m$, $\varphi\in C^\infty_p(\mR^m;\mR^k)$,
we have 
\begin{align}\label{Chain}
\varphi(F)\in (\mW^{1,\infty-}_\Theta)^k\  \mbox{ and }\ D_\Theta\varphi (F)=D_\Theta F^i\p_i\varphi(F).
\end{align}

\end{enumerate}
\et

The following Kusuoka and Stroock's formula is proven in \cite[Proposition 2.11]{So-Zh}.
\bp\label{Pr1}
Fix $\Theta=(h,\v)\in\mH_{\infty-}\times\mV_{\infty-}$.  Let $\eta(\omega,s,z):\Omega\times[0,1]\times\Gamma_0\to\mR$ and
$f(\omega,s):\Omega\times[0,1]\to\mR^d$ be measurable maps
and satisfy that for each $(s,z)\in[0,1]\times\Gamma_0$, 
$$
\eta(s,z), f(s)\in\mW^{1,\infty-}_\Theta,\ \ \eta(s,\cdot)\in C^1(\Gamma_0),
$$ 
and $s\mapsto f(s), D_\Theta f(s)$ are $\sF_s$-adapted, 
\begin{align}
\mbox{$s\mapsto  \eta(s,z), D_\Theta\eta(s,z),\nabla_z\eta(s,z)$ are left-continuous and $\sF_s$-adapted}.
\end{align}
Assume that  for any $p>1$, $\int^1_0\|f(s)\|_{\Theta;1,p}^p\dif s<\infty$ and
\begin{align}
\mE\left[\sup_{s\in[0,1]}\sup_{z\in\Gamma_0}\left(\frac{|\eta(s,z)|^p+|D_\Theta\eta(s,z)|^p}{(1\wedge|z|)^p}
+|\nabla_z\eta(s,z)|^p\right)\right]<+\infty.\label{ET7}
\end{align}
Then $\sI_1(f):=\int^1_0f(s)\dif W_s$, $\sI_2(\eta):=\int^1_0\!\int_{\Gamma_0}\eta(s,z)\widetilde N(\dif s,\dif z)\in\mW^{1,\infty-}_\Theta$ and
\begin{align}\label{For2}
\begin{split}
D_\Theta \sI_1(f)&=\int^1_0 D_\Theta f(s)\dif W_s+\int^1_0 f(s)\dot h(s)\dif s,\\
D_\Theta \sI_2(\eta)&=\int^1_0\!\!\!\int_{\Gamma_0}D_\Theta\eta(s,z)\widetilde N(\dif s,\dif z)+\int^1_0\!\!\!\int_{\Gamma_0}\nabla_\v\eta(s,z)N(\dif s,\dif z),
\end{split}
\end{align}
where $\nabla_\v\eta(s,z):=\v_i(s,z)\p_{z_i}\eta(s,z)$.
\ep

We also need the following Burkholder's type inequalities (cf. \cite[Lemma 2.3]{So-Zh}).
\bl\label{Le2}
\begin{enumerate}[(i)]
\item For any $p>1$, there is a constant $C_p>0$ such that for any $\xi\in \mL^1_p$,
\begin{align}
\mE\left(\sup_{t\in[0,1]}\left|\int^t_0\!\!\!\int_{\Gamma_0}\xi(s,z)N(\dif s,\dif z)\right|^p\right)\leq C_p\|\xi\|^p_{\mL^1_p}.\label{BT1}
\end{align}
\item For any $p\geq 2$, there is a constant $C_p>0$ such that for any $\xi\in \mL^2_p$,
\begin{align}
\mE\left(\sup_{t\in[0,1]}\left|\int^t_0\!\!\!\int_{\Gamma_0}\xi(s,z)\widetilde N(\dif s,\dif z)\right|^p\right)\leq C_p\|\xi\|^p_{\mL^2_p}.\label{BT2}
\end{align}
\end{enumerate}
\el
The following result is taken from \cite[Lemma 2.5]{Zhang17}, which is stated in a slightly different form.
\bl\label{Le25}
Let $g_s(z),\eta_s$ be two left continuous $\sF_s$-adapted processes satisfying that for some $\beta\in(0,1]$,
\begin{align}
0\leq g_s(z)\leq\eta_s,\ |g_s(z)-g_s(0)|\leq\eta_s |z|^\beta,\ \forall |z|\leq 1,\label{UY22}
\end{align}
and for any $p\geq 2$,
$$
\mE\left(\sup_{s\in[0,1]}|\eta_s|^p\right)<+\infty.
$$
Let $f_s$ be a nonnegative measurable adapted process and $\nu(\dif z)=\dif z/|z|^{d+\alpha}$.
For any $\delta\in(0,1)$ and $m\geq 2$, there exist $c_2,\theta\in(0,1), C_2\geq 1$ such that for all $\lambda,p\geq 1$ and $t\in(0,1)$,
\begin{align}\label{TR7}
\begin{split}
&\mE\exp\left\{-\lambda\int^t_0\!\!\!\int_{\mR^d_0}g_s(z)\zeta_{m,\delta}(z)N(\dif s,\dif z)-\lambda\int^t_0f_s\dif s\right\}\\
&\quad\leq C_2\left(\mE\exp\left\{-c_2\lambda^\theta\int^t_0(f_s+g_s(0))\dif s\right\}\right)^{\frac{1}{2}}+C_p\lambda^{-p},
\end{split}
\end{align}
where $\zeta_{m,\delta}(z)$ is a nonnegative smooth function with
$$
\zeta_{m,\delta}(z)=|z|^{m},\ \ |z|\leq\delta/4,\ \ \zeta_{m,\delta}(z)=0,\ \ |z|>\delta/2.
$$
\el

\section{Improved Komatsu-Takeuchi's type estimate}

Let  $\sS_m$ be the class of all $m-$dimensional semi-martingales with the following form
\begin{eqnarray*}
 X_t=X_0+\int_0^t f^0_s\dif s+\int_0^t  f^k_s\dif W^k_s+\int_0^t \int_{\mR_0^d}g_s(z)\widetilde{N}(\dif s,\dif z),
\end{eqnarray*}
where $f^k_s,k=0,\cdots,d$ and $g_s(z)$ are $m$-dimensional predictable processes with
\begin{eqnarray*}
   \|X_\cdot(\omega)\|_{\sS_m}:=\sup_{s\in [0,1]}
   \left(|X_s(\omega)|^2\vee |f^0_s(\omega)|^2\vee |f^k_s(\omega)|^2 \vee \sup_{z\in \mR^d}\frac{|g_s(z,\omega)|^2}{1\wedge |z|^2}\right)<\infty \mbox{ for }a.a.-\omega.
\end{eqnarray*}
Here and below we use the following convention: If an index appears twice in a product, then it will be summed automatically. For instances,
$$
\int_0^t  f^k_s\dif W^k_s:=\sum_{k=1}^d\int_0^t  f^k_s\dif W^k_s,\ \ |f_s^k(\omega)|^2:=\sum_{k=1}^d|f_s^k(\omega)|^2.
$$
For $\kappa>0$, let $\sS_m^\kappa$ be the subclass of $\sS_m$ with
\begin{eqnarray*}
  \|X_\cdot(\omega)\|_{\sS_m}\leq  \kappa  \mbox{ for } a.a.-\omega.
\end{eqnarray*}

We first recall the following estimates from \cite[Theorem 4.2]{Zhang16}.
\bl\label{Le42}
For $\kappa>0$, let $(X_t)_{t\geq 0}$ and $(f^0_t)_{t\geq 0}$ be two semimartingales  in $\sS_m^\kappa$ with the form
\begin{align*}
X_t&=X_0+\int^{t\wedge\tau}_0(f^0_s+h^\delta_s)\dif s+\int^t_0f^k_s\dif W^k_s+\int^t_0\!\!\!\int_{|z|\leq\delta}g_{s-}(z)\widetilde N(\dif s,\dif z),\\
f^0_t&=f^0_0+\int^t_0f^{00}_s\dif s+\int^t_0f^{0k}_s\dif W^k_s+\int^t_0\!\!\!\int_{|z|\leq\delta}g^0_{s-}(z)\widetilde N(\dif s,\dif z),
\end{align*}
where $\delta\in(0,1]$ and $\tau$ is a stopping time.
Assume that for some $\beta\geq 0$,
$$
|h^\delta_t|^2\leq\kappa\delta^\beta,\ \ a.s.
$$
For any $\eps,\delta, t\in(0,1]$, there are positive random variables $\zeta_1,\zeta_2$ with $\mE \zeta_i\leq 1, i=1,2$ such that  
almost surely
\begin{eqnarray}
 \label{AA-11}
c_0\int_0^t\left(|f_s^k|^2+\int_{|z|\leq \delta }|g_s(z)|^2 \nu(\dif z)\right)\dif s
\leq (\delta^{-1}+\eps^{-1})\int_0^t|X_s|^2\dif s +\kappa \delta\log \zeta_1+\kappa(\eps+t\delta),
 \end{eqnarray}
 and
\begin{align}
c_1\int^{t\wedge\tau}_0|f^0_s|^2\dif s\leq(\delta^{-\frac{3}{2}}+\eps^{-\frac{3}{2}})\int^t_0|X_s|^2\dif s
+\kappa\delta^{\frac{1}{2}}\log \zeta_2+\kappa(\eps\delta^{-\frac{1}{2}}+\eps^{\frac{1}{2}}+t\delta^{\frac{1}{2}\wedge\beta}),\label{Eso2}
\end{align}
where $c_0,c_1\in(0,1)$ only depend on $\int_{|z|\leq 1}|z|^2\nu(\dif z)$.
\el
Below we show a refinement  of \eqref{AA-11}  and \cite[Lemma 5.1]{KT2001} for our aim.
\begin{lemma}
\label{8-1}
For $\kappa>0$ and $\delta\in(0,1)$, let  $(X_t)_{t\geq 0}$ be a semi-martingale  in $\sS_m^\kappa$ with the form
$$
X_t=X_0+\int_0^{t}f^0_s\dif s+\int_0^t f_s^k \dif W_s^k+\int_0^t\!\! \int_{|z|\leq \delta }g_s(z)\widetilde{N}(\dif s,\dif z).
$$
For any $\eps,\delta,t\in(0,1)$, there is a positive random variable $\zeta$ with $\mE\zeta\leq 1$ such that for all $\beta\in[0,2]$,
it holds almost surely
 \begin{eqnarray}
 \label{AA-1}
 \begin{split}
&c_0\int_0^t\left(|f_s^k|^2+\int_{|z|\leq \delta }|g_s(z)|^2 \nu(\dif z)\right)\dif s
\leq (\delta^{-\beta}+\eps^{-1})\int_0^t|X_s|^2\dif s +\kappa \delta^\beta\log \zeta+\kappa \eps,
 \end{split}
 \end{eqnarray}
where $c_0\in(0,1)$ only depends on $\int_{|z|\leq 1}|z|^2\nu(\dif z)$.
In particular, if  $\nu(\dif z)=\dif z/|z|^{d+\alpha}$ for some  $\alpha\in (0,2)$ and
\begin{eqnarray}\label{z-4}
g_s(z)=\Gamma_s\cdot z+\widetilde g_s(z)\mbox{ with }
\sup_{s\in [0,1]}\left( \|\Gamma_s\|^2_{HS} \vee  \sup_{|z|\leq 1}\frac{|\widetilde g_s(z)|^2}{|z|^4 }\right) \leq \kappa,\quad a.e.-\omega,
\end{eqnarray}
where $\Gamma_s: \mR_+\times\mR^d\to\mR^m\otimes\mR^d$ is a matrix valued predictable process,
 and $\|\cdot\|_{HS}$ denotes the Hilbert-Schdmit norm,  then for some $c_0=c_0(d,\alpha)\in(0,1)$,
 \begin{eqnarray}
 \label{AA-10}
 \begin{split}
c_0\int_0^t\Big(|f_s^k|^2+\|\Gamma_s\|^2_{HS}\Big)\dif s\leq \delta^{2\alpha-6}
\int_0^t|X_s|^2\dif s +\kappa \delta^\alpha\log \zeta+\kappa \delta^2.
  \end{split}
 \end{eqnarray}
\end{lemma}
\begin{proof}
By replacing $(X_t,f_t^k,g_t(z))$  with $(X_t,f_t^k,g_t(z))/\sqrt{\kappa}$, we may assume that
\begin{align}\label{Bo}
|X_t|^2\vee |f^0_t|^2\vee |f^k_t|^2\vee\frac{|g_t(z)|^2}{1\wedge|z|^2}\leq 1.
\end{align}
Notice that for $t\leq 2\eps$,
\begin{eqnarray*}
\int_0^t |f_s^k|^2 \dif s+  \int_0^t\!\!\int_{|z|\leq \delta }|g_s(z)|^2  \nu(\dif z)\dif s\leq 2\eps\left(1+\int_{|z|\leq 1}|z|^2\nu(\dif z)\right).
\end{eqnarray*}
This implies  (\ref{AA-1}) with $\zeta\equiv 1$.
Below, without loss of generality, we assume $t>2\eps$ and $m=1.$ 
Following the proof in \cite{KT2001}, let $\rho(u):\mR_+\to\mR_+$ be an increasing smooth function with
\begin{align}\label{Rho}
\mbox{$\rho(u)=u$ for $0\leq u<4$ and $\rho(u)=\tfrac{9}{2}$ for  $u>6$,  $\rho'(u)\leq 1$.}
\end{align}
Since $\rho(u)\leq u$ for $u\geq 0$, letting $\eps_s:=(t-\eps)\wedge s-0\vee (s-\eps)$, we have
\begin{align*}
\delta^{-\beta}\int_0^t |X_s|^2\dif s&\geq \int_0^t\rho(\delta^{-\beta}|X_s|^2)\dif s
\geq\int^t_\eps\rho(\delta^{-\beta}|X_s|^2)\dif s-\int^{t-\eps}_0\rho(\delta^{-\beta}|X_s|^2)\dif s\\
&=\int_0^{t-\eps}\dif s\int_s^{s+\eps }\dif \rho(\delta^{-\beta} |X_r|^2 )
=\int_0^t\eps_s\dif \rho(\delta^{-\beta} |X_s|^2),
\end{align*}
 where  the second equality is due to Fubini's theorem.
By It\^o's formula, we obtain
\begin{align*}
\delta^{-\beta}\int_0^t|X_s|^2\dif s&\geq 2\delta^{-\beta}\int_0^t \eps_s\rho'_sX_s f^0_s\dif s+\delta^{-\beta}\int_0^t \eps_s\big\{\rho'_s+2\delta^{-\beta}|X_s|^2\rho''_s\big\}|f^k_s|^2\dif s\\
&+\left\{2\delta^{-\beta}\int_0^t\eps_s\rho'_sf _sf^k_s\dif W^k_s+\int_0^t\!\!\int_{|z|\leq \delta}\eps_s(\tilde\rho_s(z)-\rho_s)\widetilde{N}(\dif s,\dif z)\right\}\\
&+\int_0^t\!\!\int_{|z|\leq \delta}\eps_s\Big((\tilde\rho_s(z)-\rho_s )
 -2\delta^{-\beta} \rho'_s X_s  g_s(z) \Big)\nu(\dif z)\dif s=:\sum_{i=1}^4I_i(t),
\end{align*}
where $\tilde\rho_s(z):=\rho(\delta^{-\beta} |X_s+g_s(z)|^2)$ and
$$
\rho_s:=\rho(\delta^{-\beta} |X_s|^2),\ \ \rho'_s:=\rho' (\delta^{-\beta} |X_s|^2),\ \ \rho''_s:=\rho'' (\delta^{-\beta} |X_s|^2).
$$
For $I_1(t)$, thanks to $|\eps_s|\leq\eps,$ by \eqref{Bo} and \eqref{Rho} we have
 $$
|I_1(t)|\leq \delta^{-\beta} \int_0^t\big(|X_s|^2+ |\eps_sf^0_s|^2\big)\dif s  \leq   \delta^{-\beta}\int_0^t|X_s|^2\dif s+\delta^{-\beta}\eps^2 t.
 $$
For $I_2(t)$, noticing that
 \begin{eqnarray} \label{B-2}
|\eps_s-\eps|\leq \eps \big\{1_{(0,\eps)}(s)+1_{(t-\eps,t)}(s)\big\}
 \end{eqnarray}
and
\begin{eqnarray}  \label{B-3}
  \eps_s\rho'_s=\eps+\eps( \rho'_s-1)+(\eps_s-\eps)\rho'_s\geq \eps-\eps\delta^{-\beta}|X_s|^2-|\eps_s-\eps|,
\end{eqnarray}
where we have used $|\rho'(u)-1|\leq u$ and $|\rho'(u)|\leq 1$, we have
\begin{align*}
  I_2(t)&\geq\delta^{-\beta}\int_0^t\eps_s \rho'_s|f^k_s|^2   \dif s-2\|\rho''\|_\infty\delta^{-2\beta}\int_{0}^t\eps_s|X_s|^2|f^k_s|^2\dif s\\
  &\geq \delta^{-\beta}\int_0^t\Big\{\eps-\eps\delta^{-\beta}|X_s|^2-|\eps_s-\eps|\Big\}|f^k_s|^2\dif s-2\|\rho''\|_\infty\delta^{-2\beta}\eps\int_{0}^t|X_s|^2\dif s\\
  &\geq  \eps\delta^{-\beta}\int_0^t|f^k_s|^2  \dif s-(1+2\|\rho''\|_\infty)\eps\delta^{-2\beta}\int_{0}^t|X_s|^2\dif s-\delta^{-\beta}\int^t_0|\eps_s-\eps|\dif s\\
  &=  \eps\delta^{-\beta}\int_0^t|f^k_s|^2  \dif s-(1+2\|\rho''\|_\infty)\eps\delta^{-2\beta}\int_{0}^t|X_s|^2\dif s-2\delta^{-\beta}\eps^2.
\end{align*}
For $I_3(t)$, noticing that
$$
|\eps_s(\tilde\rho(z)-\rho )|\leq \eps\Big\{9\wedge (\delta^{-\beta}|\,|X_s+g_s(z)|^2-|X_s|^2|)\Big\},
$$
by \cite[Lemma 4.1]{Zhang16} with $R=\frac{1}{9\eps}$, there is a  positive random variable $\zeta_1$  with $\mE\zeta_1\leq 1$ so that
\begin{align*}
  -I_3(t)-9\eps\log \zeta_1
&  \leq \frac{2\delta^{-2\beta}}{9\eps}\int_0^t|\eps_s\rho'_s X_s|^2|f^k_s|^2\dif s
  +\frac{2}{9\eps}\int_0^t\!\! \int_{|z|\leq \delta}\eps_s^2(\tilde\rho_s(z)-\rho_s)^2\nu(\dif z)\dif s
  \\ &  \stackrel{\eqref{Bo}}{\leq} \frac{2\eps\delta^{-2\beta}}{9}\left[\int_0^t|X_s|^2 \dif s
  +\int_0^t \int_{|z|\leq \delta}\big(|X_s+g_s(z)|^2-X_s^2\big)^2 \nu(\dif z)\dif s\right]
   \\ &  \leq \frac{2\eps\delta^{-2\beta}}{9}\left[\left(1+4\int_{|z|\leq\delta}\!\!|z|^2\nu(\dif z)\right)\int_0^t|X_s|^2 \dif s
  +\int_0^t\!\!\int_{|z|\leq \delta}|g_s(z)|^4 \nu(\dif z)\dif s\right]
   \\ & \stackrel{\eqref{Bo}}{\leq} C\eps\delta^{-2\beta}\int_0^t|X_s|^2 \dif s
  +\frac{2\eps\delta^{2-2\beta}}{9}\int_0^t\!\!\int_{|z|\leq \delta}|g_s(z)|^2 \nu(\dif z)\dif s.
\end{align*}
For $I_4(t)$, noticing that by Taylor's expansion for $x\mapsto\rho(\delta^{-\beta}x)$,
\begin{align*}
&\rho(\delta^{-\beta}|X+g|^2)-\rho(\delta^{-\beta}|X|^2)-2\delta^{-\beta}Xg\rho'(\delta^{-\beta}|X|^2)\\
&=\delta^{-\beta}\rho'(\delta^{-\beta}|X|^2)g^2+\tfrac{1}{2}\delta^{-2\beta}\rho''(\vartheta)(2Xg+|g|^2)^2,
\end{align*}
where $\vartheta:=\delta^{-\beta}(r|X|^2+(1-r)|X+g|^2)$ for some $r\in[0,1]$, we can write
   \begin{align*}
 I_4(t)& = \frac{\delta^{-2\beta} }{2}\int_0^t\!\!\int_{|z|\leq \delta }\eps_s\rho''(\vartheta_s(z))(2X_sg_s(z)+|g_s(z)|^2)^2\nu(\dif z)\dif s\\
&  +\delta^{-\beta} \int_0^t\!\!\int_{|z|\leq \delta }\eps_s \rho'_s|g_s(z)|^2\nu(\dif z)\dif s=:I_{41}(t)+I_{42}(t),
 \end{align*}
 where $\vartheta_s(z):=\delta^{-\beta}(r|X_s|^2+(1-r)|X_s+g_s(z)|^2)$.
 For $I_{41}(t)$, noticing that
 $$
|X_s|\vee|g_s(z)|\leq\delta\Rightarrow\vartheta_s(z)\leq \delta^{-\beta}(r\delta^2+4(1-r)\delta^2)\leq 4\delta^{2-\beta}\leq 4,
 $$
in virtue of $\rho''(u)=0$ for $u\leq 4$,  by \eqref{Bo} we have
\begin{align*}
|I_{41}(t)| &=\frac{\delta^{-2\beta} }{2}\left|\int_0^t\!\!\int_{|z|\leq \delta<|X_s| }\eps_s\rho''(\vartheta_s(z))(2X_sg_s(z)+|g_s(z)|^2)^2\nu(\dif z)\dif s\right|\\
&\leq \frac{\delta^{-2\beta}}{2}\eps\|\rho''\|_\infty\left(\int_{|z|\leq \delta}|z|^2\nu(\dif z)\right)\int_0^t(3|X_s|)^2\dif s.
\end{align*}
 For $I_{42}(t)$, as in the treatment of $I_2(t)$, by \eqref{B-3} we have
 \begin{align*}
 I_{42}(t) &\geq \delta^{-\beta} \int_0^t\!\!\int_{|z|\leq \delta }\Big\{\eps-\eps\delta^{-\beta}|X_s|^2-|\eps_s-\eps|\Big\}|g_s(z)|^2\nu(\dif z)\dif s\\
  &\geq \eps\delta^{-\beta} \int_0^t\!\!\int_{|z|\leq \delta }|g_s(z)|^2\nu(\dif z)\dif s- \delta^{-\beta}\int_0^t(\eps\delta^{-\beta}|X_s|^2+|\eps_s-\eps|)\dif s
  \int_{|z|\leq\delta}|z|^2\nu(\dif z)\\
&\geq \eps\delta^{-\beta}\int_0^t\!\!\int_{|z|\leq \delta }|g_s(z)|^2\nu(\dif z)\dif s-\left(\eps\delta^{-2\beta} \int_0^t|X_s|^2\dif s+2\eps^2\delta^{-\beta}\right)\int_{|z|\leq1}|z|^2\nu(\dif z).
 \end{align*}
Combining the above estimates, we get for some $C\geq 9$ only depending on $\int_{|z|\leq 1}|z|^2\nu(\dif z)$,
 \begin{align*}
   & \delta^{-\beta}\Big\{2+C\eps\delta^{-\beta}\Big\} \int_0^t|X_s|^2\dif s+9\eps\log \zeta_1+C\eps^2\delta^{-\beta}
   \\ & \geq \eps\delta^{-\beta}\int_0^t|f^k_s|^2   \dif s+\left(\eps\delta^{-\beta}-\tfrac{2}{9}\eps\delta^{2-2\beta}\right)\int_0^t\!\!\int_{|z|\leq \delta}|g_s(z)|^2\nu(\dif z)\dif s,
 \end{align*}
which gives \eqref{AA-1} with $\zeta=\zeta_1^{9/C}$
by dividing both sides by $C\eps\delta^{-\beta}$.

If $\nu(\dif z)=\dif z/|z|^{d+\alpha}$, then by $g_s(z)=\Gamma_s z+\widetilde g_s(z)$ and (\ref{z-4}), we have
\begin{eqnarray}\label{pp-1}
\begin{split}
  \int_0^t\!\! \int_{|z|\leq \delta }|g_s(z)|^2\nu(\dif z)
  &\geq \frac{1}{2} \int_0^t\!\! \int_{|z|\leq \delta }|\Gamma_s z|^2\nu(\dif z)
  - \int_0^t\!\! \int_{|z|\leq \delta }|\widetilde g_s(z)|^2\nu(\dif z)
    \\ & \geq  c\delta^{2-\alpha}\int_0^t\|\Gamma_s\|^2_{HS}\dif s
  -\kappa \int_0^t\!\! \int_{|z|\leq \delta }|z|^4\nu(\dif z) \\
  &\geq  c\delta^{2-\alpha}\int_0^t \|\Gamma_s\|^2_{HS}\dif s- C\kappa\delta^{4-\alpha},
  \end{split}
\end{eqnarray}
where in the second inequality we have used that
$$
\int_{|z|\leq \delta}z_iz_j\dif z/|z|^{d+\alpha}=c {\bf 1}_{i=j}\delta^{2-\alpha}.
$$
Substituting this into (\ref{AA-1}) and taking $\beta=2$ and $\eps=\delta^{4-\alpha}$, we get 
$$
c_0\int_0^t|f_s^k|^2\dif s+c\delta^{2-\alpha}\int_0^t \|\Gamma_s\|^2_{HS}\dif s\leq 
(\delta^{\alpha-4}+\delta^{-2})\int_0^t|X_s|^2\dif s +\kappa \delta^2\log \zeta+\kappa\delta^{4-\alpha},
$$
which gives \eqref{AA-10} by dividing both sides by $\delta^{2-\alpha}$.
\end{proof}
\br
If we take $\beta=1$ and $\eps=\delta$ in \eqref{AA-1}, then \eqref{AA-1} reduces to \eqref{AA-11} with $\eps=\delta$.
Here the key point for us is the estimate \eqref{AA-10}, which can not be derived from \eqref{AA-11}.
\er

\section{Estimates of Laplace transform of reduced Malliavin matrix}

The result of this section is independent of the framework in Section 2.
Consider the following SDE with jumps:
\begin{eqnarray}\label{SDE0}
  X_t=x+\int_0^tb(X_s)\dif s+\int_0^t\sigma_k(X_s)\dif W^k_s+\int_0^t\!\!\int_{|z|<1}g(X_{s-},z)\widetilde{N}(\dif s,\dif z).
\end{eqnarray}
It is well known that under {\bf (H$_1$)}, SDE \eqref{SDE0} admits a unique solution denoted by $X_t=X_t(x)$, and $x\mapsto X_t(x)$ are smooth. 
Let $J_t:=J_t(x):=\nabla X_t(x)$ be the Jacobian matrix of $X_t(x)$. It is also well known that $J_t$ solves the following linear matrix-valued SDE:
\begin{eqnarray}\label{Q22}
\begin{split}
  J_t&=\mI+\int_0^t \nabla b(X_s)J_s\dif s+\int_0^t \nabla \sigma_k(X_s)J_s\dif W^k_s+\int_0^t\!\!\int_{|z|<1}\nabla_xg(X_{s-},z)J_{s-}\widetilde{N}(\dif s,\dif z).
  \end{split}
\end{eqnarray}
Moreover, under {\bf (H$_g^{\rm o}$)},  the matrix $J_t(x)$ is invertible. Let
$K_t:=K_t(x)$ be the inverse matrix of $J_t(x).$ By It\^o's formula, $K_t$ solves the following linear matrix-valued SDE:
\begin{align}\label{Q2}
\begin{split}
  K_t&=\mI-\int_0^tK_s\big[\nabla b(X_s)-(\nabla \sigma_k)^2(X_s)\big]\dif s-\int_0^tK_s \nabla\sigma_k(X_s)\dif W^k_s
  \\ &+\int_0^t\!\!\int_{|z|<1}K_{s-}Q(X_{s-},z)\widetilde{N}(\dif s,\dif z)
-\int_0^t\!\!\int_{|z|<1}K_{s-}(Q\cdot\nabla_xg)(X_{s-},z)\nu(\dif z)\dif s,
\end{split}
\end{align}
 where
   \begin{eqnarray*}
Q(x,z):= (\mI+\nabla_xg(x,z))^{-1}-\mI=-(\mI+\nabla_xg(x,z))^{-1}\cdot\nabla_xg(x,z).
 \end{eqnarray*}
 
Below we introduce some notations for later use.
\begin{enumerate}[$\bullet$]
\item For $k=1,\cdots,d$, let $A_0, A_k,\widetilde A_k$ be defined as in the introduction.
For simplicity, we write
\begin{align}\label{KP5}
A:=(A_1,\cdots, A_d),\ \ \widetilde A:=(\widetilde A_1,\cdots,\widetilde A_d).
\end{align}
\item For a vector field $V(x):={\rm v}_i(x)\p_i$, we also identify
$$
V:=({\rm v}_1,\cdots,{\rm v}_d).
$$
\item For a smooth vector field $V:\mR^d\to\mR^d$, define 
$$
\overline{V}:=[A_0,V]+\tfrac{1}{2}[A_k,[A_k,V]],
$$
and for $\delta\in(0,1)$,
\begin{align}\label{TG1}
\begin{split}
&G_V(x,z):=V(x+g(x,z))-V(x)+Q(x,z)V(x+g(x,z)),\\
&H^\delta_V(x):=\int_{|z|\leq \delta }\big[G_V(x,z)+\nabla_xg(x,z)\cdot V(x)-g(x,z)\cdot \nabla V(x)\big]\nu(\dif z).
\end{split}
\end{align}
 \item For a row vector $u\in\mR^d$ and a smooth vector field $V:\mR^d\to\mR^d$, define 
\begin{align}\label{HH0}
\begin{split}
 \sH_t(u,x)&:=\int_0^t|uK_s(x)V(X_s(x))|^2\dif s,\\
\sW_t(u,x)&:=\int_0^t|uK_s(x) ([A,V],[\widetilde A,V],\overline{V})(X_s(x))|^2 \dif s.
\end{split}
\end{align}
\end{enumerate}

We need the following easy lemma.
 \bl\label{Le22}
Let $V\in C^\infty_p(\mR^d;\mR^d)$. Under {\bf (H$_1$)} and {\bf (H$_g^{\rm o}$)}, we have
\begin{enumerate}[(i)]
\item For any $\delta\in(0,1)$, there are $m\in\mN_0$ and $C>0$ such that for all $x\in\mR^d$ and $|z|\leq 1$,
$$
|Q(x,z)|\leq C|z|,\ \ |G_V(x,z)|\leq C(1+|x|^m)|z|,\ \ |H^\delta_V(x)|\leq C(1+|x|^m).
$$
\item $[A_k,V],[\widetilde A_k,V], \overline V\in C^\infty_p(\mR^d;\mR^d)$.
\item $\nabla_z  G_V(x,0)=[\widetilde{A},V](x).$
\end{enumerate}
 \el
\begin{proof}
We only prove (iii). Let $R(x,z):=(\mI+\nabla_xg(x,z))^{-1}$. Note that
$$
G_V(x,z)=R(x,z) V(x+g(x,z))-V(x)
$$
and
$$
\nabla_zG_V(x,0)=\nabla_z R(x,0) V(x+g(x,0))+R(x,0)\nabla V(x+g(x,0))\nabla_zg(x,0).
$$
Since $g(x,0)=0$, $R(x,0)=\mI$ and $\nabla_z R(x,0)=-\nabla_x\nabla_zg(x,0)$, we get
$$
\nabla_zG_V(x,0)=\nabla V(x)\nabla_zg(x,0)-\nabla_x\nabla_zg(x,0)V(x).
$$
The proof is complete.
\end{proof}
\br
Under {\bf (H$_1'$)} and $V\in C^\infty_b(\mR^d;\mR^d)$, the $m$ in the above (i) can be zero.
\er
The following lemma is standard by BDG's inequality (see Lemma \ref{Le2}). We omit the details.
 \begin{lemma}\label{15}
   Under  {\bf (H$_1$)} and {\bf (H$_g^{\rm o}$)}, for any $p\geq 1,$ we have
   \begin{eqnarray*}
     \sup_{x\in \mR^d}\mE\left(\sup_{t\in [0,1]} \left(\frac{|X_t(x)|^p}{1+|x|^p}+|J_t(x)|^p+|K_t(x)|^p\right)\right)<+\infty.
   \end{eqnarray*}
 \end{lemma}
Now we can show the following crucial lemma.
\begin{lemma}\label{p-6}
Let $\beta:=\frac{\alpha\wedge(2-\alpha)}{8}$ and $d_0\in\mN$. Under  {\bf (H$_1$)} and {\bf (H$_g^{\rm o}$)},
for any $C_p^\infty$-function $V:\mR^d\rightarrow \mR^d\otimes\mR^{d_0}$,
there exist constants $c\in (0,1),C\geq 1$ such that for all $\delta,t\in(0,1),x\in \mR^d$ and $p\geq 1$,
\begin{eqnarray}
\sup_{|u|=1}\mP\Big\{\sH_t(u,x)\leq \delta^7 t,  \sW_t(u,x)\geq \delta^\beta t\Big\}\leq C_p(x)\delta^p+C\e^{-c\delta^{-\beta}t},\label{12}
\end{eqnarray}
where $\sH_t(u,x)$ and $\sW_t(u,x)$ are defined by \eqref{HH0}.
Moreover, under {\bf (H$_1'$)} and $V\in C^\infty_b$, the constant $C_p(x)$ can be independent of $x.$
\end{lemma}
\begin{proof}
We divide the proof into four steps.

\textbf{(1)}
Fixing $\delta\in(0,1)$, we make the following decomposition:
$$
L_t:=\int^t_0\!\!\int_{|z|\leq\delta}z\widetilde N(\dif s,\dif z)+\int^t_0\!\!\int_{\delta<|z|<1}zN(\dif s,\dif z)=:L^\delta_t+\hat L^\delta_t,
$$
where $L^\delta$ and $\hat L^\delta$ are the small jump part and  large jump part of $L$, respectively. Clearly,
$$
\mbox{$L^\delta_\cdot$ and $\hat L^\delta_\cdot$ are independent}.
$$
Let us fix a c\'adl\'ag path $\hbar:\mR_+\to\mR^d$ with finitely many jumps on the finite time interval. 
Let $X_t^\delta(x;\hbar)$ solve the following SDE:
\begin{align*}
X_t^\delta(x;\hbar)&=x+\int_0^tb(X_s^\delta(x;\hbar))\dif s+\int_0^t A_k(X_s^\delta(x;\hbar))\dif W_s^k
\\ & +\int_0^t\!\!\int_{|z|\leq \delta}g(X_{s-}^\delta(x;\hbar),z)\widetilde{N}(\dif s,\dif z)+\sum_{0<s\leq t}g(X_{s-}^\delta(x;\hbar),\Delta\hbar_s).
\end{align*}
Let $K^\delta_t(x;\hbar):=[\nabla X^\delta_t(x;\hbar)]^{-1}$. Clearly, by $g(x,-z)=-g(x,z)$, we have
\begin{align}
X_t(x)=X^\delta_t(x;\hbar)|_{\hbar=\hat L^\delta},\ \ K_t(x)=K^\delta_t(x;\hbar)|_{\hbar=\hat L^\delta},\label{EH2}
\end{align}
which implies that
\begin{align}\label{Def1}
\sH_t(u,x)=\sH_t^\delta(u,x;\hbar)|_{\hbar=\hat L^\delta},\ \sW_t(u,x)=\sW_t^\delta(u,x;\hbar)|_{\hbar=\hat L^\delta},
\end{align}
where for a row vector $u\in\mR^d$,
\begin{align}\label{Def0}
\begin{split}
\sH_t^\delta(u,x;\hbar)&:=\int_0^t|uK^\delta_s(x;\hbar)V(X^\delta_s(x;\hbar))|^2\dif s,\\
\sW_t^\delta(u,x;\hbar)&:=\int_0^t|uK^\delta_s(x;\hbar) ([A,V],[\widetilde A,V],\overline V)(X^\delta_s(x;\hbar))|^2 \dif s.
\end{split}
\end{align}

{\bf (2)} Below, we first consider the case $\hbar=0$. For simplicity, we drop the superscript $\delta$ and  write 
$$
K_t=K_t(x;0),\ \ X_t=X_t(x;0).
$$ 
Moreover, for fixed $u\in\mR^d$, we introduce the following processes: for $k=1,\cdots,d$,
\begin{align*}
  & f_t:=uK_t V(X_t),\ \
  f_t^0:=uK_t \overline{V}(X_t),\ \   h_t:=uK_t H^\delta_V(X_t),
  \\ & f_t^k:= uK_t [A_k,V](X_t),\quad g_t(z):=uK_{t-} G_V(X_{t-},z),\\
    &    f_t^{0k}:= uK_t[A_k,\overline{V}](X_t),\quad g_t^0(z):=uK_{t-} G_{\overline{V}}(X_{t-},z),
     \\  &   f_t^{00}:=uK_t 
     \Big([A_0,\overline{V}]+\tfrac{1}{2}[A_j,[A_j,\overline{V}]] + H^\delta_{\overline{V}}\Big)(X_t),
   \end{align*}
where $H^\delta_V$ and $G_V$ are defined by \eqref{TG1}.
Note that $K_t$ solves  the following equation (see \eqref{Q2}):
\begin{align*}
  K_t&=\mI-\int_0^tK_s \big[\nabla b(X_s)-(\nabla A_k)^2(X_s)\big]\dif s-\int_0^tK_s  \nabla  A_k(X_s)\dif W_s^k
  \\ & +\int_0^t\!\!\int_{|z|\leq \delta}K_{s-} Q(X_{s-},z)\widetilde{N}(\dif s,\dif z)
 -\int_0^t\!\!\int_{|z|\leq \delta}K_{s-} (Q\cdot\nabla_xg)(X_{s-},z)\nu(\dif z)\dif s.
\end{align*}
Using It\^o's formula, one finds that
\begin{align}\label{IY}
\begin{split}
 f_t&=uV(x)+\int_0^t (f_s^0+h_s) \dif s+\int_0^t f_s^k\dif W_s^k +\int_0^t\!\!\int_{|z|\leq \delta}g_s(z)\widetilde{N}(\dif s,\dif z),\\
f_t^0&= u\overline{V}(x)+\int_0^tf^{00}_s\dif s +\int^t_0f^{0k}_s\dif W_s^k+\int_0^t\!\! \int_{|z|\leq \delta}g^0_s(z)
\widetilde{N}(\dif s\dif z).
\end{split}
\end{align}
Let $\Gamma_t:=uK_t (\nabla_z  G_V)(X_t,0)$. By (iii) of Lemma \ref{Le22}, we have
$$
\Gamma_t=uK_t[\widetilde A, V](X_t).
$$ 
Since $V\in C^\infty_p$, by Lemma \ref{Le22}, there exist an $m\in \mN_0$ and $C\geq 1$ independent of $u,x\in\mR^d$ 
such that for all $t\in[0,1]$, $|z|\leq 1$ and $k=0,\cdots,d$,
\begin{eqnarray}
\label{ppp-1}
\begin{split}
  & |f_t|+|f_t^k|+|f_t^{0k}|+\|\Gamma_t\|_{HS}\leq C|uK_t|(1+|X_t|^m),
\\ &   |g_t(z)|+|g_t^0(z)|\leq C|uK_t|(1+|X_t|^m)|z|,
\\ &   |\widetilde g_t(z)|:=|g_t(z)-\Gamma_tz|\leq C|uK_t|(1+|X_t|^m)|z|^2,
\end{split}
\end{eqnarray}
and
\begin{eqnarray}
\label{4-1}
\begin{split}
|h_t | & \leq  |uK_t | \int_{|z|\leq \delta}|H^\delta_V(X_t,z)|\nu(\dif z)
\\ &\leq C |uK_t | \int_{|z|\leq \delta} (1+|X_t|^m)|z|^2\nu(\dif z)
\\ &\leq  C  |uK_t |(1+|X_t|^m) \delta^{2-\alpha} .
\end{split}
\end{eqnarray}
 For $\gamma:=\frac{\alpha\wedge(2-\alpha)}{4}$ and $m$ being as above, define a stopping time
\begin{eqnarray*}
  \tau:=\Big\{s>0: |uK_s|^2\vee |X_s|^{2m}>\delta^{-\gamma/2}\Big \}.
\end{eqnarray*}
By (\ref{ppp-1}) and (\ref{4-1}),   there exists $\kappa_0>0$ such that  for all $t\in [0,\tau)$,  $|z|\leq 1$ and  $k=0,\cdots,d$
\begin{eqnarray}
  \nonumber & |f_t|^2+|f_t^k|^2+|f_t^{0k}|^2+|\Gamma_t|^2\leq \kappa_0 \delta^{-\gamma}, \quad &  |g_t(z)|^2+|g_t^0(t)|^2\leq \kappa_0 \delta^{-\gamma}|z|^2,
  \\ \label{2-1} &  |\widetilde g_t(z)|^2  \leq  \kappa_0 \delta^{-\gamma}|z|^4, & |h_t|^2 \leq \kappa_0\delta^{4-2\alpha -\gamma}.
\end{eqnarray}
Moreover, by \eqref{IY} we also have
\begin{align*}
 f_{t\wedge\tau}&=uV(x)+\int_0^{t\wedge\tau} (f_s^0+h_s) \dif s+\int_0^{t} 1_{s<\tau}f_s^k\dif W_s^k 
 +\int_0^{t}\!\!\int_{|z|\leq \delta}1_{s<\tau}g_s(z)\widetilde{N}(\dif s,\dif z),\\
f_{t\wedge\tau}^0&= u\overline{V}(x)+\int_0^{t}1_{s<\tau}f^{00}_s\dif s +\int^{t}_01_{s<\tau}f^{0k}_s\dif W_s^k
+\int_0^{t}\!\! \int_{|z|\leq \delta}1_{s<\tau}g^0_s(z)\widetilde{N}(\dif s\dif z).
\end{align*}
Fix $t\in(0,1)$. By \eqref{AA-10} with  $\kappa=\kappa_0\delta^{-\gamma}$, 
there are constant $c_0=c_0(d,\alpha)\in(0,1)$ and random variable $\zeta_1>0$ with $\mE\zeta_1\leq 1$ such that
\begin{eqnarray}
 \label{3-4}
c_0\int_0^{t\wedge\tau}(|f_s^k|^2+\|\Gamma\|_{HS}^2)\dif s
\leq \delta^{2\alpha-6}\int_0^t|f_{s\wedge\tau}|^2\dif s + \kappa_0 \delta^{\alpha-\gamma}\log \zeta_1+\kappa_0 \delta^{2-\gamma}.
 \end{eqnarray}
By  (\ref{Eso2})  with $\eps=\delta^4,$ $\kappa=\kappa_0\delta^{-\gamma}$ and $\beta=4-2\alpha$, 
 there are constant $c_1=c_1(d,\alpha)>0$ and  random variable $\zeta_2>0$ with $\mE\zeta_2\leq 1$ such that
 \begin{eqnarray}
  \label{3-3}
c_1 \int_0^{t\wedge \tau} | f_s^0|^2\dif s \leq 
2\delta^{-6}\int_0^t |f_{s\wedge\tau}|^2\dif s +\kappa_0 \delta^{\frac{1}{2}-\gamma} \log \zeta_2
  + \kappa_0\delta^{-\gamma}(2\delta^2 +t\delta^{\frac{1}{2}\wedge (4-2\alpha)}).
 \end{eqnarray}
Combining \eqref{3-4}, \eqref{3-3} with definition \eqref{Def0}, one finds that for some $c_2=c_2(\kappa_0,d,\alpha)>0$,
 $$
c_2 \sW_t^\delta(u,x;0)\leq \delta^{-6} \sH_t^\delta(u,x;0)+\tfrac{1}{2}\log(\zeta_1^{\delta^{\alpha-\gamma}}\zeta^{\delta^{\frac{1}{2}-\gamma}}_2)+\delta^{2-\gamma}+t\delta^{\frac{1}{2}\wedge (4-2\alpha)-\gamma},\ t<\tau.
 $$
Multiplying both sides by $\delta^{-\gamma}$ and taking exponential, we obtain
 $$
{\bf 1}_{\{t<\tau\}}\exp\left\{c_2 \delta^{-\gamma}\sW_t^\delta(u,x;0)-\delta^{-6-\gamma} \sH_t^\delta(u,x;0)\right\}\leq (\zeta_1^{\delta^{\alpha-2\gamma}}\zeta^{\delta^{\frac{1}{2}-2\gamma}}_2)^{\frac{1}{2}}\e^{\delta^{2-2\gamma}+t\delta^{\frac{1}{2}\wedge (4-2\alpha)-2\gamma}}.
 $$
Recalling $\gamma=\frac{\alpha\wedge(2-\alpha)}{4}$, taking expectations and by $\mE\zeta_1\leq 1$ and $\mE\zeta_2\leq 1$,  we derive 
 \begin{eqnarray}
\label{3-1}
\sup_{u,x\in \mR^d}\mE\Big({\bf 1}_{\{t<\tau\}} \exp\big\{c_2\delta^{-\gamma}\sW_t^\delta(u,x;0)-\delta^{-7}\sH_t^\delta(u,x;0)\big\}\Big) 
\leq \e^{\delta^\alpha+t}, \ \forall\delta,t\in(0,1).
\end{eqnarray}

 {\bf (3)} Let $m$ be as in \eqref{ppp-1}. We introduce the following random set for later use:
  \begin{eqnarray*}
 \Omega^\delta_t(u,x;\hbar):=\left\{\sup_{s\in[0,t]}\Big(|uK_s^\delta(x;\hbar)|^2\vee |X_s^\delta(x;\hbar)|^{2m}\Big)\leq \delta^{-\gamma/2}\right\}.
  \end{eqnarray*}
We use \eqref{3-1} to show that there is a $C\geq 1$ such that for all $t,\delta\in(0,1)$ and $u,x\in\mR^d$,
  \begin{eqnarray}
   \label{111}
\mathbf{II}:=\mE \Big({\bf 1}_{\Omega_t^\delta(u,x;\hat L^\delta)}\cdot\exp\big\{c_2\delta^{-\gamma}\sW_t(u,x)-\delta^{-7}\sH_t(u,x)\big\}\Big) \leq C.
\end{eqnarray}
Let
$$ 
 \cJ_t^u(x;\hbar):={\bf 1}_{\Omega_t^\delta(u,x;\hbar)}\cdot\exp\Big\{c_2\delta^{-\gamma}\sW_t(u,x;\hbar)-\delta^{-7}\sH_t(u,x;\hbar)\Big\}.
 $$
 Since $ \Omega^\delta_t(u,x;0)\subset\{t<\tau\}$, by (\ref{3-1}) we have
 \begin{eqnarray}\label{10}
\sup_{u,x\in \mR^d}\mE \cJ_t^u(x;0)\leq \e^{\delta^\alpha+ t}.
 \end{eqnarray}
 Let $0=t_0<t_1<\cdots<t_n\leq t_{n+1}=t$ be the jump times of $\hbar$ before time $t$.
 For $j=0,1,\cdots,n$ and $s\in[0,t_{j+1}-t_j)$, noting that
$$  
X_{s+t_j}^\delta(x;\hbar)= X_{s}^\delta(X_{t_j}^\delta(x;\hbar);0) 
\Rightarrow K_{s+t_j}^\delta(x;\hbar)=K_{t_j}^\delta(x;\hbar)K_s^\delta(X_{t_j}^\delta(x;\hbar);0),
$$
we have
\begin{eqnarray*}
  \Omega_{t_{j+1}}(u,x;\hbar)= \Omega_{t_{j}}(u,x;\hbar)\cap \Omega_{t_{j+1}-t_j}\Big(uK_{t_j}^\delta(x;\hbar),X_{t_j}^\delta(x;\hbar);0\Big).
\end{eqnarray*}
Thus, by the Markov property and (\ref{10}), we have for all $u,x\in \mR^d,$
\begin{align*}
  \mE \cJ_{t_{n+1}}^u(x;\hbar)&=\mE \Big( \cJ_{t_{n}}^u(x;\hbar)\cdot \mE  \cJ_{t_{n+1}-t_n}^{u_n}(x_n;0)|_{u_n=uK_{t_j}^\delta(x;\hbar),x_n=X_{t_n}^\delta(x;\hbar)}\Big)
  \\ &\leq \e^{\delta^\alpha+t_{n+1}-t_n}  \mE  \cJ_{t_{n}}^u(x;\hbar)\leq\cdots\leq \e^{ n\delta^\alpha+ t}.
\end{align*}
Now we can give a proof of  (\ref{111}).
Let
\begin{eqnarray*}
 N_t^\delta:=\sum_{s\in [0,t]}{\bf 1}_{|\Delta \hat L_s|>\delta}, \quad  \lambda_\delta:=\int_{\delta<|z|<1}\dif z/|z|^{d+\alpha}.
\end{eqnarray*}
Then $N^\delta$ is a Poisson process with intensity $\lambda_\delta$.
By the independence of $L^\delta$ and $\hat{L}^\delta$, we have
  \begin{align*}
\mathbf{II}&\stackrel{\eqref{Def1}}{=}\mE \cJ_t^u(x;\hat{L}^\delta)
 =\sum_{n=0}^\infty \mE \big(\mE \cJ_t^u(x;\hbar)|_{h=\hat{L}^\delta};N_t^\delta=n\big)
    \\&\leq \sum_{n=0}^\infty \e^{  n\delta^{\alpha } +t} \mP(N_t^\delta=n)
=\e^{t}\sum_{n=0}^\infty  \e^{ n\delta^{\alpha }}  \frac{(t\lambda_\delta)^n}{n!}\e^{-t\lambda_\delta}
    \\ &=\e^t\exp{\{ ( \e^{ \delta^{\alpha }}-1) \lambda_\delta t\}}
\leq \e^t\exp{\{ 3 \delta^{\alpha} \lambda_\delta \}}\leq C,
  \end{align*}
where we have used that $\e^s-1\leq 3s $ for $s\in (0,1)$ and  $\lambda_\delta\leq c\delta^{-\alpha}$.

\medskip

{\bf (4)} For any  $p>1,$ by Chebyshev's inequality we have
\begin{align}\label{11}
\begin{split}
\mP((\Omega^\delta_t(u,x;\hat{L}^\delta))^c)
&=\mP\left(\sup_{s\in[0,t]}\Big(|uK_s(x)|^2\vee |X_s(x)|^{2m}\Big)>\delta^{-\gamma/2} \right) \\ 
& \leq \delta^p \mE\left[\sup_{s\in[0,t]}\Big(|uK_s(x)|^2\vee |X_s(x)|^{2m}\Big)^{2p/\gamma}\right] \\ 
&\leq (C_p(x)+C|u|^{\frac{4p}{\gamma}}) \delta^p.
\end{split}
\end{align}
Therefore, for all $t\in (0,1),x\in \mR^d, |u|=1$ and $p\geq 1$, by \eqref{EH2}, \eqref{11} and \eqref{111}, we have
\begin{eqnarray*}
&& \mP\left\{\sH_t(u,x)\leq \delta^7 t, \sW_t(u,x)\geq \delta^{\gamma/2}t\right\} \leq \mP\big((\Omega^\delta_t(u,x;\hat{L}^\delta))^c\big)\\
&&+ \mP\left\{c_2\delta^{-\gamma}\sW_t(u,x)-\delta^{-7}\sH_t(u,x)\geq c_2\delta^{-\gamma/2}t-t; \Omega^\delta_t(u,x;\hat{L}^\delta)\right\}
  \\ &&\leq (C_p(x)+C|u|^{\frac{4p}{ \gamma}}) \delta^p+C\e^{t-c_2\delta^{-\gamma/2 }t}.
\end{eqnarray*}
We complete   the proof of  (\ref{12}) by setting $\beta=\frac{\gamma}{2}=\frac{\alpha\wedge(2-\alpha)}{8}.$

\medskip

{\bf (5)}  Under {\bf (H$_1'$)} and $V\in C^\infty_b$,  
the $m$ in Lemma \ref{Le22} and \eqref{ppp-1} can be zero so that  $C_p(x)$ can be independent of $x.$
\end{proof}

By a standard chain argument, we have the following lemma, which is the same in spirit as \cite[Lemma 3.1]{Kuni}.
\bl\label{Le34}
Under {\bf (H$_1$)}, 
for any $n\in\mN$, there is a constant $C_n\geq 1$ such that for all $R\geq 1$,
\begin{align}\label{KP9}
\sup_{|x|<R}\mP\left(\sup_{s\in[0,\eps]}|X_s(x)|\geq C_nR\right)\leq C_n\eps^n,\ \ \forall \eps\in(0,1).
\end{align}
\el
\begin{proof}
Let $\mD$ be the space of all c\'adl\'ag functions from $[0,1]$ to $\mR^d$. Let $\bP_x=\mP\circ X^{-1}_\cdot(x)$ be the law of $X_\cdot(x)$ in $\mD$.
With a little of confusion, let $X$ be the coordinate process over $\mD$ so that $(X,\bP_x)_{x\in\mR^d}$ forms a family of strong Markov processes.
Let $\tau_0=0$ and $R\geq 1$. For $j\in\mN$, define
$$
\tau_j:=\inf\left\{t>\tau_{j-1}: |X_t-X_{\tau_{j-1}}|>R\right\},\ L_t:=\int^t_0\!\!\int_{|z|<1}z\widetilde N(\dif s,\dif z).
$$
Clearly, we have
$$
\sup_{t\in[\tau_{j-1},\tau_j)}|X_t-X_{\tau_{j-1}}|\leq R,\quad X_{\tau_j}=X_{\tau_j-}+g(X_{\tau_j-},\Delta L_{\tau_j}).
$$
Since $|g(x,z)|\leq C(1+|x|)$ for all $x\in\mR^d$ and $|z|<1$, we have
$$
|X_{\tau_j}|\leq |X_{\tau_j-}|+C(1+|X_{\tau_j-}|)\leq C_1R+C_2|X_{\tau_{j-1}}|.
$$
By induction method, one sees that for each $j=1,2,\cdots$,
$$
|X_{\tau_j}|\leq \frac{(C_2^{j}-1)C_1}{C_2-1}R+C_2^j|X_0|\leq C_nR,\ |X_0|\leq R,
$$
and therefore,
\begin{align}\label{GQ1}
\sup_{s\in[0,\tau_n]}|X_s|\leq C_nR,\ |X_0|\leq R,
\end{align}
which implies that
\begin{align}\label{KJ7}
\left\{\sup_{s\in[0,\eps]}|X_s|>C_n R\right\}\subset\Big\{\tau_n<\eps\Big\},\ \ |X_0|\leq R.
\end{align}
Noting that $\tau_n=\tau_{n-1}+\tau_1\circ \theta_{\tau_{n-1}}$, where $\theta_\cdot$ is the usual shift operator,
by the  strong Markov property, we have
\begin{align}
\bP_x\Big\{\tau_n<\eps\Big\}&=\bP_x\Big\{\tau_n<\eps;\tau_{n-1}<\eps\Big\}\leq\bP_x\Big\{\tau_1\circ\theta_{\tau_{n-1}}<\eps;\tau_{n-1}<\eps\Big\}\no\\
&=\bE_x\left(\bP_{X_{\tau_{n-1}}}\Big(\tau_1<\eps\Big);\tau_{n-1}<\eps\right).\label{GQ3}
\end{align}
On the other hand, by BDG's inequality and the linear growth of $b,\sigma$ and $g$, we have
\begin{align*}
\bE_y\left(\sup_{t\in[0,\eps]}|X_t-X_0|^2\right)&\lesssim\bE_y\left(\int^\eps_0|b(X_s)|\dif s\right)^2+\bE_y\left(\int^\eps_0|\sigma(X_s)|^2\dif s\right)\\
&\quad+\bE_y\left(\int^\eps_0\!\!\int_{|z|<1}|g(X_s,z)|^2\nu(\dif z)\dif s\right)\\
&\lesssim\eps\sup_{s\in[0,\eps]}\bE_y(1+|X_s|^2)\lesssim\eps(1+|y|^2).
\end{align*}
Hence,
$$
\bP_y\Big(\tau_1<\eps\Big)\leq \bP_y\left(\sup_{t\in[0,\eps]}|X_t-X_0|>R\right)\leq R^{-2}\bE_y\left(\sup_{t\in[0,\eps]}|X_t-X_0|^2\right)
\leq C R^{-2}(1+|y|^2)\eps.
$$
Thus, by \eqref{GQ1} and \eqref{GQ3}, we get for $|x|<R$,
$$
\bP_x\Big\{\tau_n<\eps\Big\}\leq CR^{-2}(1+C^2_nR^2)\eps\bP_x\Big\{\tau_{n-1}<\eps\Big\}\leq\cdots\leq C_n\eps^n.
$$
The proof is complete by \eqref{KJ7}.
\end{proof}
\br
Note that \eqref{KP9} can not be obtained by simple applications of Chebyshev's inequality. Intuitively, 
consider the Poisson process
$N_t$ with intensity $\lambda$. For $n\in\mN$, we clearly have
$$
\mP(N_t\geq n)=\e^{-\lambda t}\sum_{k=n}^\infty \frac{(\lambda t)^k}{k!}
=(\lambda t)^n\e^{-\lambda t}\sum_{k=0}^\infty \frac{(\lambda t)^k}{(k+n)!}\leq (\lambda t)^n.
$$
If we let $\tau_n$ be the $n$-th jump time of $N$, then $\{N_t\geq n\}=\{\tau_n\leq t\}$.
\er
We define the reduced Malliavin matrix by
\begin{eqnarray*}
\hat{\Sigma}_t(x):=\int_0^tK_s(x)\big(A A^*+\widetilde A\widetilde A^*\big)(X_s(x))K^*_s(x)\dif s,
  \end{eqnarray*}
  where $A$ and $\widetilde A$ are defined in \eqref{KP5}. 
 Following the proof of \cite[Theorem 3.3]{Zhang17}, we have
\begin{theorem}\label{Th41}
 Under {\bf (H$_1$)}, {\bf (H$_g^{\rm o}$)} and {\bf (H$^{\rm str}_{\bf or}$)}, there exist $\gamma=\gamma(\alpha,j_0) \in (0,1)$ 
 and constants $C_2\geq 1,c_2\in (0,1)$ 
 such that for all $t\in (0,1)$ and $R,\lambda,p\geq 1$,
 \begin{eqnarray*}
 \sup_{|x|<R}  \sup_{|u|=1}\mE\[\exp\{-\lambda u \hat{\Sigma}_t(x) u^*\}\]\leq C_2\exp\{-c_2t\lambda^\gamma\}+C_{R,p}(\lambda t)^{-p},
 \end{eqnarray*}
 where $C_{R,p}>0$ continuously depends on $R,p.$
\end{theorem}
\begin{proof}
Let $A_0,A_k$ be as in \eqref{VEC} and $\widetilde A_k:=\p_{z_k}g^i(\cdot,0)\p_i$. Define
$$
\sU_0:=\{A_k,\widetilde A_k, k=1,\cdots,d\},
$$ 
and for $j=1,2,\cdots,$
$$
   \sU_j:=\Big\{[A_k,V], [\widetilde A_k,V],[A_0,V]+\tfrac{1}{2}[A_i,[A_i,V]]:   V\in\sU_{j-1},k=1,\cdots,d\Big\}.
$$
Recall the definition of $\sV_j$ in \eqref{VJ}. It is easy to see that 
\begin{align}\label{KP7}
{\rm span}\{\cup_{j=0}^{j_0}\sV_j\}\subset{\rm span}\{\cup_{j=0}^{j_0+1}\sU_j\}.
\end{align}
Let $\gamma:=\alpha\wedge(2-\alpha)/56$ and $x,u\in\mR^d$ with $|u|=1$. 
For $j=0,1,\cdots,j_0+1$, define
$$
E_j^x:=\left\{\sum_{V\in\sU_j}\int^t_0|uK_s(x) V(X_s(x))|^2\dif s\leq t\eps^{7\gamma^j }\right\},
$$
Notice that
$$
E_0^x\subset\left(\cap_{j=0}^{j_0+1}E_j^x\right)\cup\left(\cup_{j=0}^{j_0-1}(E_j^x\setminus E_{j+1}^x)\right)
$$
and
$$
E_{j+1}^x=\left\{\sum_{V\in\sU_j}\int^t_0|uK_s(x)([A,V],[\widetilde A,V],\overline V)(X_s(x))|^2\dif s\leq t\eps^{7\gamma^{j+1} }\right\}.
$$
By \eqref{12} with $V\in \sU_j$ and $\delta=\eps^{\gamma^j}$, we have
$$
\mP(E_j^x\setminus E_{j+1}^x)\leq C_p(x)\eps^p+C\exp\{-ct\eps^{-8\gamma^{j+1}}\}.
$$
To estimate $\mP(\cap_{j=0}^{j_0+1}E_j^x)$, note that
\begin{align}
\cap_{j=0}^{j_0+1}E^x_j
&\subset\left\{\sum_{j=0}^{j_0+1}\sum_{V\in\sU_j}\int^t_0|uK_s(x) V(X_s(x))|^2\dif s\leq t\sum_{j=0}^{j_0+1}\eps^{7\gamma^j }\right\}\no\\
&\subset\left\{\sum_{j=0}^{j_0+1}\sum_{V\in\sU_j}\int^t_0|uK_s(x) V(X_s(x))|^2\dif s\leq tj_0\eps^{7\gamma^{j_0+1} }\right\}.\label{KJ2}
\end{align}
By {\bf (H$^{\rm str}_{\bf or}$)} and \eqref{KP7}, for each $x\in\mR^d$, we have
$$
\inf_{|u|=1}\sum_{j=0}^{j_0+1}\sum_{V \in \sU_j}|uV(x)|^2>0,
$$ 
which implies that for any $\kappa\geq 1$, there is a $c_0>0$ such that
\begin{align}\label{KJ11}
\inf_{|x|\leq \kappa}\inf_{|u|=1}\sum_{j=0}^{j_0+1}\sum_{V \in \sU_j}|uV(x)|^2\geq c_0.
\end{align}
Let $C_n$ be as in Lemma \ref{Le34}. Define stopping times
$$
\tau_1^x:=\inf\{t>0: |X_t(x)|\geq C_n R\},\ \tau_2^x:=\inf\{t>0: \|J_t(x)\|_{HS}\geq\eps^{-7\gamma^{j_0+1}/8}\}.
$$
Noticing that for $s<\tau^x_2$,
$$
\eps^{7\gamma^{j_0+1}/8}\leq\|J_s(x)\|_{HS}^{-1}\leq |uK_s(x)|,
$$
 by \eqref{KJ11} with $\kappa=C_n R$, we have
on $\{\tau_1^x>t\eps^{7\gamma^{j_0+1}/4}\}\cap\{\tau_2^x>t\}$,
$$
\sum_{j=0}^{j_0+1}\sum_{V\in\sU_j}\int^t_0|uK_s(x) V(X_s(x))|^2\dif s\geq c_0\int^{t\eps^{7\gamma^{j_0+1}/4}}_0|uK_s(x)|^2\dif s\geq c_0 t\eps^{7\gamma^{j_0+1}/2}.
$$
Thus, by \eqref{KJ2}, we have for $\eps\leq\eps_0$ small enough,
$$
\Big(\cap_{j=0}^{j_0+1}E^x_j\Big)\cap\{\tau_1^x>t\eps^{7\gamma^{j_0+1}/4}\}\cap\{\tau_2^x>t\}=\emptyset.
$$
On the other hand, by Lemma \ref{Le34}, we have
$$
\sup_{|x|<R}\mP\left(\tau_1^x\leq t\eps^{7\gamma^{j_0+1}/4}\right)\leq C_n(t\eps^{7\gamma^{j_0+1}/4})^n,
$$
and by Lemma \ref{15} and Chebyshev's inequality, for any $p\geq 1$,
$$
\sup_{x\in\mR^d}\mP\left(\tau_2^x\leq t\right)\leq C_p(\eps^{7\gamma^{j_0+1}/8})^p.
$$
Combining the above calculations, we obtain
\begin{align}\label{KJ33}
\sup_{|x|<R}\mP(E^x_0)\leq  C_{R,p}\eps^p+C\exp\{-ct\eps^{-8\gamma^{j_0+2}}\}+C_n(t\eps^{7\gamma^{j_0+1}/4})^n.
\end{align}
Noting that
$$
u \hat{\Sigma}_t(x) u^*=\sum_{V\in\sU_0}\int^t_0|uK_s(x) V(X_s(x))|^2\dif s,
$$
by \eqref{KJ33} with $n\geq 4p/(7\gamma^{j_0+1})$, we obtain
$$
\sup_{|x|<R}\sup_{|u|=1}\mP\Big(u \hat{\Sigma}_t(x) u^*\leq t\eps^{7\gamma^j}\Big)\leq  C_{R,p}\eps^p+C\exp\{-ct\eps^{-8\gamma^{j_0+1}}\}.
$$
The desired estimate of the Laplace transform of  $u \hat{\Sigma}_t(x) u^*$ follows from this estimate (see \cite{Zhang17}).
\end{proof}
\br\label{Re54}
Under {\bf (H$'_1$)},  {\bf (H$_g^{\rm o}$)} and {\bf (H$^{\rm uni}_{\bf or}$)}, from the above proofs, one sees that $C_{R,p}$ can be independent of $R$.
\er

\section{Proof of Theorem \ref{Th1}}
For $p,R\in[1,\infty]$ and $\varphi\in C_c^\infty(\mR^d)$, define
$$
\|\varphi\|_{p;R}:=\left(\int_{B_R}|\varphi(x)|^p\dif x\right)^{1/p},\ \cT^0_t \varphi(x):=\mE \varphi(X_t(x)).
$$
We first prepare the following lemma for later use.
\bl\label{Le51}
Under {\bf (H$_1$)} and {\bf (H$_g^{\rm o}$)}, for any $k,m\in\mN_0$ and $R,p\geq1$, there exists a constant $C_R=C(R,p,k)$ such that for all $t\in(0,1)$ and $\varphi\in C^\infty_c(\mR^d)$,
\begin{align}\label{KJ99}
\|\nabla^k\cT^0_t\nabla^m \varphi\|_{p;R}\leq C_R\sum_{j=0}^{k}\|\nabla^{m+j} \varphi\|_{p}.
\end{align}
Under {\bf (H$'_1$)} and {\bf (H$_g^{\rm o}$)}, the above $R$ can be $\infty$ so that the global estimate holds.
\el
\begin{proof}
(i) We first show \eqref{KJ99} for $k=m=0$. By the change of variables, we have
$$
\|\cT^0_t\varphi\|^p_{p;R}=\mE\int_{B_R}|\varphi(X_t(x))|^p\dif x=\int_{\mR^d}|\varphi(y)|^p\,\mE\Big(1_{B_R}(X_t^{-1}(y))\det(\nabla X^{-1}_t(y))\Big)\dif y.
$$
Noticing that
$$
\nabla X^{-1}_t(y)=(\nabla X_t)^{-1}\circ X_t^{-1}(y)=K_t\circ X_t^{-1}(y),
$$
we have
$$
\|\cT^0_t\varphi\|_{p;R}\leq\mE\left(\sup_{|x|<R}\det(K_t(x))\right)\|\varphi\|_p.
$$
On the other hand, from \eqref{SDE0} and \eqref{Q2}, it is by now standard to show that for any $p\geq 2$,
$$
\mE |K_t(x)-K_t(y)|^p\leq C|x-y|^p,\ \ x,y\in\mR^d,
$$
which implies that
$$
\mE |\det K_t(x)-\det K_t(y)|^p\leq C|x-y|^p,\ \ x,y\in\mR^d.
$$
Hence, by Kolmogorov's theorem, 
$$
\mE\left(\sup_{|x|<R}\det K_t(x)\right)\leq C_R.
$$
Thus \eqref{KJ99} holds for $k=m=0$.

(ii) Next for $k\in\mN$ and $m\in\mN_0$, by the chain rule, we have
\begin{align}\label{KJ10}
\nabla^k\cT^0_t\nabla^m \varphi(x)=\nabla^k\mE\Big((\nabla^m\varphi)(X_t(x))\Big)
=\sum_{j=0}^k\mE\Big((\nabla^{m+j}\varphi)(X^x_t)G_j\Big),
\end{align}
where $\{G_j, j=0,\cdots,k\}$ are real polynomial functions of $\nabla X_t(x),\cdots, \nabla^k X_t(x)$. Hence, by H\"older's inequality,
\begin{align}\label{HA1}
\|\nabla^k\cT^0_t\nabla^m \varphi\|_{p;R}^p\leq
\sum_{j=0}^k\left(\int_{B_R}\mE|(\nabla^{m+j}\varphi)(X_t(x))|^p\dif x\right)\sup_{x\in B_R}\left(\mE|G_j|^{p/(p-1)}\right)^{p-1}.
\end{align}
By \eqref{Q22}, it is now standard to show that for any $q\geq 1$ and $\ell\in\mN$,
\begin{align}\label{HA2}
\sup_{x\in \mR^d}\mE|\nabla^\ell X^x_t|^q<\infty.
\end{align}
Estimate \eqref{KJ99} follows by \eqref{HA1}, \eqref{HA2} and (i).

(iii) Under {\bf (H$'_1$)} and {\bf (H$_g^{\rm o}$)}, by \cite[Lemma 4.4]{Zhang17}, we have $\|\cT^0_t\varphi\|_p\leq C\|\varphi\|_p$, which together with 
\eqref{HA1} and \eqref{HA2} implies \eqref{KJ99} with $R=\infty$.
\end{proof}

Now we use the Malliavin calculus introduced in Section 2 to show the following main result (see also \cite{Zhang17}), which will automatically produce the conclusions 
in Theorem \ref{Th1} by Sobolev's embedding theorem.
\bt\label{Th45}
Fix $\ell\geq 2$. Under {\bf (H$_\ell$)}, {\bf (H$_g^{\rm o}$)} and {\bf (H$^{\rm str}_{\bf or}$)}, 
for any $k,m\in\mN_0$ with $k+m=\ell-1$, there exists a $\gamma_{km}>0$ such that for all $R\geq 1$, $t\in(0,1)$, $p\in(1,\infty]$
and $\varphi\in C^\infty_c(\mR^d)$,
\begin{align}\label{KJ9}
\|\nabla^k\cT^0_t\nabla^m \varphi\|_{p;R}\leq C_{R,p} t^{-\gamma_{km}}\|\varphi\|_{p}.
\end{align}
Moreover, under {\bf (H$'_\ell$)}, {\bf (H$_g^{\rm o}$)} and {\bf (H$^{\rm uni}_{\bf or}$)},  one can take $R=\infty$ in \eqref{KJ9}.
\et

First of all, by the chain rule and Proposition \ref{Pr1}, we have the following Malliavin differentiability of $X_t$ in the sense of Theorem \ref{Th21}.
Since the proof is completely the same as in \cite{Zhang17}, we omit the details.
\bl
Let $\Theta=(h,\v)\in\mH_{\infty-}\times \mV_{\infty-}$. Under {\bf (H$_1$)}, for any $t\in[0,1]$, $X_t\in\mW^{1,\infty-}_\Theta(\Omega)$  and
\begin{align}\label{SDE7}
\begin{split}
D_\Theta X_t&=\int^t_0\nabla b(X_s)D_\Theta X_s\dif s+\int^t_0\nabla \sigma_k(X_s)D_\Theta X_s\dif W^k_s+\int^t_0\sigma_k(X_s) \dot{h}^k_s\dif s\\
&+\int^t_0\!\!\!\int_{|z|<1}\nabla_x g(X_{s-},z)D_\Theta X_{s-}\widetilde N(\dif s,\dif z)
+\int^t_0\!\!\!\int_{|z|<1}\nabla_\v g(X_{s-},z) N(\dif s,\dif z),
\end{split}
\end{align}
where $\nabla_\v g(x,z):=\p_{z_i}g(x,z)\v_i(s,z)$.
Moreover, for any $R,p\geq 2$, we have
\begin{align}
\sup_{|x|<R}\mE\left(\sup_{t\in[0,1]}|D_\Theta X_t(x)|^p\right)<\infty.\label{HH2}
\end{align}
Under {\bf (H$'_1$)}, the above estimate holds for $R=\infty$.
\el

To use the integration by parts formula in Section 2, we need to introduce suitable Malliavin matrix.
Let $J_t=J_t(x)=\nabla X_t(x)$ be the Jacobian matrix of $x\mapsto X_t(x)$, and $K_t(x)$ be the inverse of $J_t(x)$. 
Recalling (\ref{Q22}) and (\ref{SDE7}), by the formula of constant variation, we have
\begin{align}
D_\Theta X_t=J_t\int^t_0K_s\sigma_k(X_s)\dot{h}^k_s \dif s+
J_t\int^t_0\!\!\!\int_{|z|<1}K_s\nabla_\v g(X_{s-},z) N(\dif s,\dif z).\label{EQ1}
\end{align}
Here the integral is the Lebesgue-Stieltjes integral. 

Next we want to choose special directions $\Theta_j\in\mH_{\infty-}\times\mV_{\infty-}, j=1,\cdots,d$ so that
the Malliavin matrix $\cM^{ij}_t(x):=(D_{\Theta_j}X^i_t(x))_{i,j=1,\cdots,d}$ is invertible. Let
$$
H(x;t):=\int^t_0\sigma^*(X_s(x))K^*_s(x)\dif s,
$$
where the asterisk stands for the transpose of a matrix, and
$$
U(x,z):=(\mI+\nabla_xg(x,z))^{-1}\nabla_zg(x,z),\ \ x\in\mR^d,\ \ |z|<1.
$$
The following lemma is a direction consequence of the above definition and {\bf (H$_\ell$)}.
\bl\label{Le53}
Under {\bf (H$_\ell$)}, for any $k\in\mN_0$ and $m=0,\cdots,\ell-1$, there exists $C>0$ such that for all $x\in\mR^d$ and $0<|z|<1$,
\begin{align}
|\nabla^k_x\nabla^m_zU(x,z)|\leq C(1+|x|)|z|^{-m},\ \ |U(x,z)-U(x,0)|\leq C(1+|x|)|z|^\beta,\label{TYR1}
\end{align}
where $\beta$ is the same as in {\bf (H$_\ell$)}.
\el

Below we fix $\ell\geq 2$ and assume {\bf (H$_\ell$)}$+${\bf (H$_g^{\rm o}$)}. For $j=1,\cdots,d$, define
$$
h_j(x;t):=H(x;t)_{\cdot j},\ \ \v_j(x;s,z):=[K_{s-}(x)U(X_{s-}(x),z)]^*_{\cdot j}\zeta_{\ell,\delta}(z),
$$
where $\zeta_{\ell,\delta}(z)$ is a nonnegative smooth function with
$$
\zeta_{\ell,\delta}(z)=|z|^{1+\ell},\ \ |z|\leq \delta/4,\ \ \zeta_{\ell,\delta}(z)=0,\ \ |z|>\delta/2.
$$
Let 
$$
\Theta_j(x):=(h_j(x), \v_j(x)).
$$
Noticing that by equation (\ref{Q2}),
$$
K_s=K_{s-}(\mI+\nabla_xg(X_{s-},\Delta L_s))^{-1},
$$
by (\ref{EQ1}) we have
\begin{align}
\cM^{ij}_t(x):=D_{\Theta_j} X^i_t(x)=(J_t(x)\Sigma_t(x))_{ij},\label{Var}
\end{align}
where $\Sigma_t(x)=\Sigma^{(1)}_t(x)+\Sigma^{(2)}_t(x)$, and
\begin{align}\label{Var1}
\begin{split}
\Sigma^{(1)}_t(x)&:=\int^t_0 K_s(x)(\sigma\sigma^*)(X_s(x))K^*_s(x)\dif s,\\
\Sigma^{(2)}_t(x)&:=\int^t_0\!\!\!\int_{|z|<1}K_{s-}(x)(UU^*)(X_{s-}(x),z)K^*_{s-}(x)\zeta(z)N(\dif s,\dif z).
\end{split}
\end{align}

By Lemma \ref{Le53} and cumbersome calculations (see \cite{Zhang17}), we have
\bl\label{Le43}
\begin{enumerate}[(i)]
\item For each $j=1,\cdots, d$ and $x\in\mR^d$, $\Theta_j(x)\in\mH_{\infty-}\times\mV_{\infty-}$.

\item For any $R,p\geq 1$, $k\in\mN_0$ and $m=0,\cdots,\ell-1$, we have
\begin{align}
&\sup_{|x|<R}\mE\left(\sup_{t\in[0,1]}\Big(|D^m_\Theta\nabla^k X_t(x)|^p+|D^m_\Theta\cM_t(x)|^p\Big)\right)<\infty,\label{TR4}\\
&\sup_{|x|<R}\mE\left(\sup_{t\in[0,1]}|D^{m}_\Theta\div(\Theta_i(x))|^p\right)<\infty,\ i=1,\cdots,d,\label{TR5}
\end{align}
where $D_\Theta:=(D_{\Theta_1},\cdots, D_{\Theta_d})$.
\item Under {\bf (H$'_\ell$)}$+${\bf (H$^{\rm o}_g$)}, the $R$ in \eqref{TR4}-\eqref{TR5} can be infinity.
\end{enumerate}
\el

Now we can give

\begin{proof}[Proof of Theorem \ref{Th45}]
We  divide the proof into three steps. Below we fix $\ell\geq 2$ and $k,m\in\mN_0$ so that $k+m=\ell-1$.

{\bf (1)} Let $\Sigma_t(x)=\Sigma^{(1)}_t(x)+\Sigma^{(2)}_t(x)$ be defined by (\ref{Var1}). In view of $U(x,0)=\nabla_zg(x,0)$, by  \eqref{TYR1},
Lemma \ref{Le25} and Theorem \ref{Th41}, there are constants $C_3\geq 1$, $c_3,\theta\in(0,1)$ and
$\gamma=\gamma(\alpha,j_0)\in(0,1)$ such that for all $t\in(0,1)$ and $R,\lambda,p\geq 1$,
\begin{align}
\sup_{|x|<R}\sup_{|u|=1}\mE\exp\left\{-\lambda u\Sigma_t(x)u^*\right\}\leq C_3\exp\{- c_3 t\lambda^\gamma\}+C_{R,p}(\lambda^\theta t)^{-p}.\label{EE88}
\end{align}
As in \cite[Lemma 5.3]{Zhang16}, for any $R,p\geq 1$, there exist constants $C_{R,p}\geq 1$ and $\gamma'=\gamma'(\alpha,j_0,d)>0$ 
such that for all $t\in(0,1)$,
\begin{align}\label{TR44}
\sup_{|x|<R}\mE \Big((\det\Sigma_t(x))^{-p}\Big)\leq C_{R,p}t^{-\gamma' p}.
\end{align}
Since $\cM^{-1}_t(x)=\Sigma^{-1}_t(x) K_t(x)$, by Lemma \ref{15}, \eqref{TR4} and \eqref{TR44}, we obtain that for all $p\geq 1$,
\begin{align}
\sup_{|x|<R}\|\cM^{-1}_t(x)\|_{L^p(\Omega)}\leq C_{R,p}t^{-\gamma'}, \ t\in(0,1).\label{EK9}
\end{align}
Under {\bf (H$'_\ell$)}$+${\bf (H$_g^{\rm o}$)}$+${\bf (H$^{\rm uni}_{\bf or}$)}, by Remark \ref{Re54} and (iii) of Lemma \ref{Le43}, 
the above $R$ can be infinity.

{\bf (2)} For $t\in(0,1)$ and $x\in\mR^d$, let $\sC^\ell_t(x)$ be the class of all polynomial functionals of 
$$
(D^m_\Theta\div\Theta)_{m=0}^{\ell-2}, \cM^{-1}_t,(D^m_\Theta\nabla^k X_t)_{k\in\mN,m=0,\cdots,\ell}, 
\big(D^m_\Theta\cM_t\big)^{\ell-1}_{m=1}, 
$$
where the starting point $x$ is dropped in the above random variables. 
By (\ref{EK9}) and Lemma \ref{Le43}, for any $H_t(x)\in \sC_t(x)$, there exists a $\gamma(H)>0$ only depending on the degree of $\cM^{-1}_t$ 
appearing in $H$ and $\alpha,j_0,d$ such that for all $t\in(0,1)$ and $p\geq 1$,
\begin{align}
\sup_{|x|<R}\|H_t(x)\|_{L^p(\Omega)}\leq C_{R,p} t^{-\gamma(H)}.\label{HHJ}
\end{align}
Under {\bf (H$'_\ell$)}$+${\bf (H$_g^{\rm o}$)}$+${\bf (H$^{\rm uni}_{\bf or}$)}, by \eqref{EK9} and (iii) of Lemma \ref{Le43}, the above $R$ can be infinity.

{\bf (3)} Since the Malliavin matrix $\cM_t=D_\Theta X_t$ is invertible,  by the chain rule \eqref{Chain}, we have
$$
D_\Theta\Big(\varphi(X_t)\Big)\cdot \cM_t^{-1}=(\nabla \varphi)(X_t).
$$
For any $Z\in\sC_t(x)$, by (\ref{Var}) and the integration by parts formula (\ref{ER88}), we have
\begin{align*}
\mE\Big((\nabla \varphi)(X_t)Z\Big)
&=\mE\Big(D_\Theta (\varphi(X_t))\cdot\cM_t^{-1}Z\Big)=\mE\Big(\varphi(X_t)Z'\Big),
\end{align*}
where 
$$
Z':=\div\Theta\cdot\cM_t^{-1}Z-D_{\Theta}(\cM_t^{-1}Z)\in\sC_t(x).
$$ 
Starting from this formula, by \eqref{KJ10} and induction, there exists $H\in\sC_t(x)$ such that
\begin{align*}
\nabla^k\mE\Big((\nabla^m\varphi)(X_t)\Big)=\mE \Big(\varphi(X_t)H\Big).
\end{align*}
Therefore,  for any $p\in(1,\infty)$, by (\ref{HHJ}), (\ref{KJ99}) and H\"older's inequality, we have
\begin{align*}
\|\nabla^k\cT^0_t\nabla^{m}\varphi\|_{p; R}&\leq \left(\int_{B_R}\left|\mE \Big(\varphi(X_t(x))H(x)\Big)\right|^p\dif x\right)^{\frac{1}{p}}\\
&\leq \left(\int_{B_R}\mE \Big(|\varphi|^p(X_t(x))\Big)\Big(\mE|H(x)|^{\frac{p}{p-1}}\Big)^{p-1}\dif x\right)^{\frac{1}{p}}\\
&\leq C_{R,p}t^{-\gamma(H)}\|\varphi\|_p,\ \ t\in(0,1).
\end{align*}
Under {\bf (H$'_\ell$)}$+${\bf (H$_g^{\rm o}$)}$+${\bf (H$^{\rm uni}_{\bf or}$)}, by \eqref{HHJ} and Lemma \ref{Le51}, the above $R$ can be infinity.
\end{proof}

\section{Proof of Theorem \ref{Th2}}

Throughout this section we assume {\bf (H$'_1$)}, {\bf (H$^{\rm uni}_{\bf or}$)} and $g\in C^\infty_b(\mR^d\times B_1^c)$.
Let $\chi:[0,\infty)\to[0,1]$ be a smooth function with $\chi(r)=1$ for $r<1$ and $\chi(r)=0$ for $r>2$.
For $\delta>0$, define
$$
\chi_\delta(r):=\chi(r/\delta),\ \ g_\delta(x,z):=g(x,z)\chi_\delta(z).
$$
Choose  $\delta$ be small enough so that $g_\delta$ satisfies {\bf (H$^{\rm o}_g$)}. Thus we can write
$$
\sA \varphi=\cL_0 \varphi+\sL \varphi,
$$
where
$$
\cL_0\varphi(x):=\frac{1}{2}A^2_k \varphi(x)+A_0\varphi(x)+{\rm p.v.}\int_{\mR^d_0}\Big(\varphi(x+g_\delta(x,z))-\varphi(x)\Big)\frac{\dif z}{|z|^{d+\alpha}},
$$
and
$$
\sL \varphi(x):=\int_{\mR^d_0}\Big(\varphi(x+g(x,z))-\varphi(x+g_\delta(x,z))\Big)\frac{\dif z}{|z|^{d+\alpha}}.
$$
Let $(\cT_t)_{t\geq 0}$ (resp. $(\cT^0_t)_{t\geq 0}$) be the semigroup associated with $\sA$ (resp. $\cL_0$). Then we have
\begin{align}\label{EQ2}
\p_t\cT_t \varphi=\sA\cT_t \varphi=\cL_0\cT_t \varphi+\sL\cT_t \varphi.
\end{align}
By Duhamel's formula, we have
\begin{align}\label{KP1}
\cT_t \varphi=\cT^0_t\varphi+\int^t_0\cT^0_{t-s}\sL \cT_s\varphi\dif s.
\end{align}
Notice that under {\bf (H$'_1$)} and {\bf (H$^{\rm uni}_{\bf or}$)}, \eqref{KJ99} and \eqref{KJ9} hold for $R=\infty$.

For $\beta\geq 0$ and $p\in(1,\infty)$, let $\mH^{\beta,p}:=(I-\Delta)^{-\frac{\beta}{2}}(L^p(\mR^d))$ be the usual Bessel potential space. 
It is well known that for any $k\in\mN$ and $p\in(1,\infty)$ (cf. \cite{St}), an equivalent norm in $\mH^{k,p}$ is given by
$$
\|\varphi\|_{k,p}=\sum_{j=0}^k\|\nabla^j \varphi\|_p.
$$ 
For a function $g:\mR^d\to\mR^d$, we introduce 
$$
T_g\varphi(x):=\varphi(x+g(x))-\varphi(x).
$$
\bl\label{Le61}
Let $m\in\mN_0$. Assume $g\in C^{m+1}_b$. For any $\theta\in(0,1)$ and $p>d/\theta$, there is a constant $C=C(p,\theta,m)>0$ such that
$$
\|T_g \varphi\|_{m,p}\leq C\|\varphi\|_{m+\theta,p}\sum_{|\alpha|\leq m}\left(\Pi_{j=1}^m\Big(1+\|\nabla^j g\|^{\alpha_j}_\infty\Big)\right)\|g\|_\infty^\theta,
$$
where $\alpha=(\alpha_1,\cdots,\alpha_m)$ and $|\alpha|=\alpha_1+\cdots+\alpha_m$.
In particular, we have
$$
\|T_g \varphi\|_{m+\beta,p}\leq C\|\varphi\|_{m+\beta+\theta,p}P(\|\nabla g\|_\infty,\cdots,\|\nabla^m g\|_\infty)
\left(1+\|\nabla^{m+1} g\|^{m+1}_\infty\right)^\beta\|g\|_\infty^\theta,
$$
where $P$ is a polynomial function of its arguments.
\el
\begin{proof}
Let $\theta\in(0,1)$. First of all, for $m=0$, we have
\begin{align*}
|T_g\varphi(x)|\leq \sup_{y\not=0}\frac{|\varphi(x+y)-\varphi(x)|}{|y|^\theta}|g(x)|^\theta.
\end{align*}
Recalling that for $p>d/\theta$ (see \cite[Lemma 5]{Mi}),
$$
\left\|\sup_{y\not=0}\frac{|\varphi(\cdot+y)-\varphi(\cdot)|}{|y|^\theta}\right\|_p\leq C\|\varphi\|_{\theta,p},
$$
we have
$$
\|T_g\varphi\|_p\leq C\|\varphi\|_{\theta,p}\|g\|_\infty^\theta.
$$
For $m\in\mN$, by the chain rule and induction, there is a constant $C>0$ such that for all $x\in\mR^d$,
$$
|\nabla^m T_g\varphi(x)|\leq C\sum_{|\alpha|\leq m}\left(\Pi_{j=1}^n\|\mI+\nabla^j g\|^{\alpha_j}_\infty\right)\sum_{j=1}^m|\nabla^j \varphi|(x+g(x)).
$$
As above one sees that for any $p>d/\theta$,
$$
\|\nabla^m T_g\varphi\|_p\leq C\sum_{|\alpha|\leq m}\left(\Pi_{j=1}^n(1+\|\nabla^j g\|^{\alpha_j}_\infty)\right)\sum_{j=1}^m\|\nabla^j \varphi\|_{\theta,p}\|g\|_\infty^\theta.
$$ 
The first estimate follows. As for the second estimate, it follows by interpolation.
\end{proof}

\bc
Let $\alpha\in(0,2)$ and $\beta\in(0,\alpha)$. For $\theta\in(0,1)$ with $\beta+\theta\in(0,\alpha)$ and $p>d/\theta$,
there is a constant $C>0$ such that for all $\varphi\in \mH^{\beta+\theta,p}$,
\begin{align}\label{KP2}
\|\sL \varphi\|_{\beta,p}\leq C\|\varphi\|_{\beta+\theta,p}.
\end{align}
Moreover, if the support of $g(x,\cdot)$ is contained in a ball $B_R$ for all $x\in\mR^d$, 
then the above estimate holds for all $\beta\geq 0$.
\ec
\begin{proof}
Notice that for any $j=0,1,\cdots$,
$$
\|\nabla^j_x g(\cdot,z)\|_\infty\leq C|z|,\ \ z\in\mR^d,
$$
and
$$
\sL\varphi(x)=\int_{|z|>\delta}\Big(\varphi(x+g(x,z))-\varphi(x+g_\delta(x,z))\Big)\frac{\dif z}{|z|^{d+\alpha}}.
$$
By Lemma \ref{Le61} we have
\begin{align*}
\|\sL \varphi\|_{\beta, p}&\leq\int_{|z|>\delta}\Big(\|T_{g(\cdot,z)}\varphi\|_{\beta,p}
+\|T_{g_\delta(\cdot,z)}\varphi\|_{\beta,p}\Big) \frac{\dif z}{|z|^{d+\alpha}}\\
&\leq C\|\varphi\|_{\beta+\theta,p}\int_{|z|>\delta}\frac{|z|^{\beta+\theta}}{|z|^{d+\alpha}}\dif z\leq C\|\varphi\|_{\beta+\theta,p},
\end{align*}
where the last step is due to $\beta+\theta<\alpha$.
If the support of $g(x,\cdot)$ is contained in a ball $B_R$ for all $x\in\mR^d$, then
$$
\sL\varphi(x)=\int_{\delta<|z|\leq R}\Big(\varphi(x+g(x,z))-\varphi(x+g_\delta(x,z))\Big)\frac{\dif z}{|z|^{d+\alpha}}.
$$
As above, \eqref{KP2} is direct by  Lemma \ref{Le61}.
\end{proof}

The following results are proven in \cite[Lemmas 5.2 and 5.3]{Zhang17}.
\bl
Let $\gamma_{10}$ and $\gamma_{01}$ be the same as in Theorem \ref{Th45}.
Under {\bf (H$'_1$)} and {\bf (H$^{\rm uni}_{\bf or}$)}, for any $p\in(1,\infty)$, $\theta,\vartheta\in[0,1)$ and $\beta\geq 0$, 
there exists a constant $C>0$ such that for all $\varphi\in C^\infty_0(\mR^d)$ and $t\in(0,1)$,
\begin{align}
&\|\cT^0_t\varphi\|_{\theta+\beta,p}\leq C t^{-\theta\gamma_{10}}\|\varphi\|_{\beta,p},\label{NJ01}\\
&\|\cT^0_t\Delta^{\frac{\vartheta}{2}}\varphi\|_p\leq C t^{-\vartheta\gamma_{01}}\|\varphi\|_p.\label{NJ1}
\end{align}
\el
Now we can show the following key estimate.
\bl
Let $\gamma_{10}, \gamma_{01}$ be as in Theorem \ref{Th45}. Fix $\vartheta\in[0,\frac{1}{\gamma_{01}}\wedge 1)$. 
Under {\bf (H$'_1$)}, {\bf (H$^{\rm uni}_{\bf or}$)} and $g\in C^\infty_b(\mR^d\times B^c_1)$, 
there exist $\beta>\alpha$, $p_0\geq 1$ and constant $C>0$ such that for all $t\in(0,1)$, $p>p_0$ and $\varphi\in C^\infty_0(\mR^d)$,
\begin{align}
\|\cT_t\Delta^{\frac{\vartheta}{2}}\varphi\|_{\beta,p}\leq Ct^{-\beta\gamma_{10}-\vartheta\gamma_{01}}\|\varphi\|_p.\label{KJ1}
\end{align}
Moreover, if the support of $g(x,\cdot)$ is contained in a ball $B_R$ for all $x\in\mR^d$, then the $\beta$ in \eqref{KJ1}
can be any positive number.
\el
\begin{proof}
Choose $\theta\in(0,\tfrac{1}{2\gamma_{10}}\wedge 1\wedge\alpha)$ and $M\in\mN$ such that
$$
M\theta<\alpha,\ \ \beta:=(M+1)\theta>\alpha.
$$ 
Let $p>d/\theta$. For $m=1,\cdots, M$,  we have
\begin{align*}
\|\cT_t\varphi\|_{(m+1)\theta,p}&\leq\|\cT^0_t\varphi\|_{(m+1)\theta,p}+\int^t_0\|\cT^0_{t-s}\sL\cT_s\varphi\|_{(m+1)\theta,p}\dif s\\
&\stackrel{\eqref{NJ01}}{\leq} Ct^{-\theta\gamma_{10}}\|\varphi\|_{m\theta,p}+C\int^t_0(t-s)^{-2\theta\gamma_{10}}\|\sL\cT_s\varphi\|_{(m-1)\theta,p}\dif s\\
&\stackrel{\eqref{KP2}}{\leq} Ct^{-\theta\gamma_{10}}\|\varphi\|_{m\theta,p}+C\int^t_0(t-s)^{-2\theta\gamma_{10}}\|\cT_s\varphi\|_{m\theta,p}\dif s\\
&\leq Ct^{-\theta\gamma_{10}}\|\varphi\|_{m\theta,p}+C\int^t_0(t-s)^{-2\theta\gamma_{10}}\|\cT_s\varphi\|_{(m+1)\theta,p}\dif s,
\end{align*}
which, by Gronwall's inequality, yields that for  all $t\in(0,1)$,
\begin{align}\label{KP3}
\|\cT_t\varphi\|_{(m+1)\theta,p}\leq Ct^{-\theta\gamma_{10}}\|\varphi\|_{m\theta, p}.
\end{align}
Thus, by the semigroup property of $\cT_t$ and iteration, we obtain that for $m=1,\cdots,M$,
\begin{align}\label{EB1}
\|\cT_{(m+1)t}\varphi\|_{(m+1)\theta,p}\leq Ct^{-\theta\gamma_{10}}\|\cT_{mt}\varphi\|_{m\theta,p}\leq\cdots
\leq Ct^{-(m+1)\theta\gamma_{10}}\|\cT_t\varphi\|_{p}.
\end{align}
On the other hand, by \eqref{KP1}, \eqref{KP2} and (\ref{NJ1}), we have
\begin{align*}
\|\cT_t\Delta^{\frac{\vartheta}{2}}\varphi\|_{p}&\leq\|\cT^0_t\Delta^{\frac{\vartheta}{2}}\varphi\|_{p}
+\int^t_0\|\cT^0_{t-s}\sL\cT_s\Delta^{\frac{\vartheta}{2}}\varphi\|_{p}\dif s\\
&\leq Ct^{-\vartheta\gamma_{01}}\|\varphi\|_{p}+C\int^t_0\|\cT_s\Delta^{\frac{\vartheta}{2}}\varphi\|_{\theta,p}\dif s\\
&\leq Ct^{-\vartheta\gamma_{01}}\|\varphi\|_{p}+C\int^{2t}_0\|\cT_s\Delta^{\frac{\vartheta}{2}}\varphi\|_{\theta,p}\dif s\\
&=Ct^{-\vartheta\gamma_{01}}\|\varphi\|_{p}+2C\int^{t}_0\|\cT_{2s}\Delta^{\frac{\vartheta}{2}}\varphi\|_{\theta,p}\dif s\\
&=Ct^{-\vartheta\gamma_{01}}\|\varphi\|_{p}+2C\int^{t}_0\|\cT_{s}\cT_s\Delta^{\frac{\vartheta}{2}}\varphi\|_{\theta,p}\dif s\\
&\stackrel{\eqref{KP3}}{\leq} Ct^{-\vartheta\gamma_{01}}\|\varphi\|_{p}+C\int^t_0s^{-\theta\gamma_{10}}\|\cT_s\Delta^{\frac{\vartheta}{2}}\varphi\|_{p}\dif s,
\end{align*}
which, by Gronwall's inequality, yields that for  all $t\in(0,1)$,
\begin{align}\label{EB2}
\|\cT_t\Delta^{\frac{\vartheta}{2}}\varphi\|_p\leq Ct^{-\vartheta\gamma_{01}}\|\varphi\|_{p}.
\end{align}
Combining \eqref{EB1} with \eqref{EB2}, we obtain \eqref{KJ1}.

Finally, if the support of $g(x,\cdot)$ is contained in a ball $B_R$ for all $x\in\mR^d$, then the $m$ in \eqref{EB1}
can be any natural number.
\end{proof}

Now we can give
\begin{proof}[Proof of Theorem \ref{Th2}]
Without loss of generality, we assume $t\in(0,1)$. 
Let $\eps\in(0,\beta-\alpha)$ and $p>d/\eps\vee p_0$.
For any $\varphi\in L^p(\mR^d)$, by \eqref{KJ1}
and Sobolev's embedding theorem, we have $(\mI-\Delta)^{\frac{\alpha+\eps}{2}}\cT_t\varphi\in C_b(\mR^d)$ 
and for any $t\in(0,1)$ and $\vartheta\in[0,\frac{1}{\gamma_{01}}\wedge 1)$,
\begin{align}
\|(\mI-\Delta)^{\frac{\alpha+\eps}{2}}\cT_t\Delta^{\frac{\vartheta}{2}}\varphi\|_\infty\leq
C\|\cT_t\Delta^{\frac{\vartheta}{2}}\varphi\|_{\beta,p}
\leq Ct^{-\beta\gamma_{10}-\vartheta\gamma_{01}}\|\varphi\|_p.\label{KG1}
\end{align}
In particular, for each $t,x$, there is a function $\rho_t(x,\cdot)\in L^{\frac{p}{p-1}}(\mR^d)$ such that for any $\varphi\in L^p(\mR^d)$,
$$
\cT_t\varphi(x)=\int_{\mR^d}\varphi(y)\rho_t(x,y)\dif y.
$$
Moreover, we also have
\begin{align*}
\sup_{x\in\mR^d}\|(\mI-\Delta)^{\frac{\alpha+\eps}{2}}_x\Delta^{\frac{\vartheta}{2}}_y\rho_t(x,\cdot)\|_{\frac{p}{p-1}}
&=\sup_{x\in\mR^d}\sup_{\varphi\in C^\infty_0(\mR^d), \|\varphi\|_p\leq 1}\left|\int\varphi(y)(\mI-\Delta)^{\frac{\alpha+\eps}{2}}_x\Delta^{\frac{\vartheta}{2}}_y\rho_t(x,y)\dif y\right|\\
&=\sup_{x\in\mR^d}\sup_{\varphi\in C^\infty_0(\mR^d), \|\varphi\|_p\leq 1}\left|\int\Delta^{\frac{\vartheta}{2}}_y\varphi(y)(\mI-\Delta)^{\frac{\alpha+\eps}{2}}_x\rho_t(x,y)\dif y\right|\\
&=\sup_{x\in\mR^d}\sup_{\varphi\in C^\infty_0(\mR^d), \|\varphi\|_p\leq 1}|(\mI-\Delta)^{\frac{\alpha+\eps}{2}}_x\cT_t\Delta^{\frac{\vartheta}{2}}\varphi(x)|\\
&=\sup_{\varphi\in C^\infty_0(\mR^d), \|\varphi\|_p\leq 1}\|(\mI-\Delta)^{\frac{\alpha+\eps}{2}}_x\cT_t
\Delta^{\frac{\vartheta}{2}}\varphi\|_\infty\leq Ct^{-\beta\gamma_{10}-\vartheta\gamma_{01}}.
\end{align*}
Thus we obtain (ii). As for (iii), it follows by \eqref{KJ1} for all $\beta\geq 0$.
\end{proof}

\section{Application to nonlocal kinetic operators}

Consider the following nonlocal kinetic operator:
$$
\sK u(x,\v)=\sL u(x,\v)+{\rm v}\cdot\nabla_x u(x,{\rm v})+b(x,{\rm v})\cdot\nabla_{\rm v} u(x,{\rm v}),
$$
where
$$
\sL u(x,\v):={\rm p.v}\int_{\mR^d}(u(x,{\rm v}+w)-u(x,{\rm v}))\frac{\kappa(x,{\rm v},w)}{|w|^{d+\alpha}}\dif w.
$$
Here $\kappa$ and $b$ satisfy the following assumptions:
\begin{enumerate}[(A)]
\item[{\bf (A)}] $b\in C^\infty_b(\mR^{2d})$, and $\kappa\in C^\infty_b(\mR^{3d})$ satisfies that for some $\kappa_0>1$,
$$
\kappa_0^{-1}\leq\kappa(x,{\rm v},w)\leq\kappa_0,\ \ \kappa(x,\v,-w)=\kappa(x,\v,w).
$$ 
Moreover, for any $i,j\in\mN_0$, there is a constant $C_{ij}>0$ such that
$$
|\nabla^{i+1}_x\nabla^j_{\rm v}\kappa(x,{\rm v},w)|\leq C_{ij}/(1+|\v|^2).
$$
\end{enumerate}
The aim of this section is to use Theorem \ref{Main} to show the following result.
\bt\label{Th71}
Under {\bf (A)}, there is a nonnegative continuous function $\rho_t(x,\v,y,w)$ on $\mR_+\times\mR^d\times\mR^d\times\mR^d\times\mR^d$
such that for each $t,y,w$, the map $(x,\v)\mapsto\rho_t(x,\v,y,w)$ belongs to $C^\infty_b(\mR^{2d})$, and for each $t>0$ and $y,w\in\mR^d$,
$$
\p_t\rho_t(x,\v,y,w)=\sK\rho_t(\cdot,\cdot,y,w)(x,\v).
$$
\et
The key point for us is the following transform lemma.
\bl\label{Le71}
Given $R\in(0,\infty)$ and $\kappa_0\in[1,\infty)$, let $\kappa(z):B_R\to[\kappa_0^{-1},\kappa_0]$ be a measurable function, where $B_R:=\{x\in\mR^d: |x|<R\}$. 
For any $\alpha\in(0,2)$, there is a homeomorphism $\Phi: B_R\to B_R$ such that for any nonnegative measurable function $f$,
\begin{align}\label{QW1}
\int_{B_R}f\circ\Phi(z)\frac{\dif z}{|z|^{d+\alpha}}=\int_{B_R}f(z)\frac{\kappa(z)}{|z|^{d+\alpha}}\dif z.
\end{align}
Moreover, we have the following properties about $\Phi$:
\begin{enumerate}[(i)]
\item $\Phi(0)=0$ and if $\kappa(-z)=\kappa(z)$, then $\Phi(-z)=-\Phi(z)$.

\item If $\kappa$ is continuous at $0$, then $\Phi$ is differentiable at point zero and
$$
\nabla\Phi(0)=\kappa(0)^{1/\alpha}\mI.
$$

\item If $\kappa\in C^1(B_R)$, then for some $C=C(\kappa_0,\|\nabla\kappa\|_\infty,\alpha, R)>0$ and any $z\in B_R$,
$$
|\nabla\Phi(z)-\nabla\Phi(0)|\leq\left\{
\begin{aligned}
&C|z|^\alpha,\ \ &\alpha\in(0,1),\\
&C|z|\log^+|z|,\ \ &\alpha=1,\\
&C|z|,&\alpha\in(1,2).
\end{aligned}
\right.
$$
\item If $\kappa\in C^j(B_R)$ for some $j\in\mN$, then for some $C_j=C_j(R)$ and any $z\in B_R$,
$$
|\nabla^j\Phi(z)|\leq C_j|z|^{1-j}.
$$
\end{enumerate}
\el
\begin{proof}
Using the spherical coordinate transform, \eqref{QW1} is equivalent to 
\begin{align}\label{NN1}
\int^R_0\!\!\!\int_{\mS^{d-1}}f\circ\Phi(t\omega)\dif\omega\frac{\dif t}{|t|^{1+\alpha}}
=\int^R_0\!\!\!\int_{\mS^{d-1}}f(t\omega)\kappa(t\omega)\dif\omega\frac{\dif t}{|t|^{1+\alpha}},
\end{align}
where $\mS^{d-1}:=\{\omega: |\omega|=1\}$ is the unit sphere in $\mR^d$.
Given $\omega\in\mS^{d-1}$, let $\phi(\cdot,\omega)$ and $\psi(\cdot,\omega)$ be defined by the following identity respectively:
\begin{align}\label{NN3}
\int^R_{\phi(r,\omega)}\frac{\dif t}{t^{1+\alpha}}=\int^R_r\frac{\kappa(t\omega)\dif t}{t^{1+\alpha}},\quad 
\int^R_r\frac{\dif t}{t^{1+\alpha}}=\int^R_{\psi(r,\omega)}\frac{\kappa(t\omega)\dif t}{t^{1+\alpha}}.
\end{align}
Since $\kappa_0^{-1}\leq\kappa(z)\leq\kappa_0$, it is easy to see that $\phi(\cdot,\omega), \psi(\cdot,\omega):[0,R]\to[0,R]$ 
are strictly increasing  continuous functions and have the following properties:
$$
\phi(0,\omega)=\psi(0,\omega)=0,\ \phi(R,\omega)=\psi(R,\omega)=R
$$
and
\begin{align}\label{NN6}
\phi(\psi(r,\omega),\omega)=\psi(\phi(r,\omega),\omega)=r.
\end{align}
Moreover, 
\begin{align}\label{NN7}
\kappa_0^{-1/\alpha} r\leq \phi(r,\omega)\leq \kappa^{1/\alpha}_0 r.
\end{align}
Indeed, by \eqref{NN3} and $\kappa^{-1}_0\leq \kappa(t\omega)\leq\kappa_0$ with $\kappa_0\geq 1$, we have
$$
\phi(r,\omega)^{-\alpha}\leq R^{-\alpha}+\kappa_0(r^{-\alpha}-R^{-\alpha})\leq\kappa_0 r^{-\alpha}\Rightarrow \phi(r,\omega)\geq\kappa_0^{-1/\alpha} r
$$
and
$$
\phi(r,\omega)^{-\alpha}\geq R^{-\alpha}+\kappa^{-1}_0(r^{-\alpha}-R^{-\alpha})\geq\kappa^{-1}_0 r^{-\alpha}\Rightarrow \phi(r,\omega)\leq\kappa_0^{1/\alpha} r.
$$
In particular, by a standard monotone class argument, it holds that for all nonnegative measurable function $g:[0,R)\to[0,\infty)$,
\begin{align}\label{NN2}
\int^R_0 g(\psi(t,\omega))\frac{\dif t}{t^{1+\alpha}}=\int^R_0 g(t)\frac{\kappa(t\omega)}{t^{1+\alpha}}\dif t.
\end{align}
Now let us define
$$
a(z):=\psi(|z|,z/|z|)/|z|,\ \ \Phi(z):=a(z)\cdot z.
$$
By \eqref{NN2}, one sees that \eqref{NN1} holds for the above $\Phi$, and by \eqref{NN6} and \eqref{NN7},
\begin{align}\label{NN9}
\kappa_0^{-1/\alpha}\leq a(z)\leq \kappa_0^{1/\alpha},\ \ \forall z\in B_R.
\end{align}
(i) From the above construction of $\Phi$, it is easy to see that $\Phi(0)=0$ and  
$$
\kappa(-z)=\kappa(z)\Rightarrow\Phi(-z)=-\Phi(z).
$$ 
(ii) We assume $\kappa(z)$ is continuous at $0$. We have the following claim:
\begin{align}\label{NN8}
a(0):=\lim_{z\to 0}a(z)=\kappa(0)^{1/\alpha},\ \nabla\Phi(0)=a(0)\mI.
\end{align}
For $i=1,\cdots,d$, let $\e_i=(0,\cdots, 1,\cdots, 0)$. Noticing that $\Phi(0)=0$, we have
$$
\p_i\Phi^j(0)=\lim_{\eps\to 0}\Phi^j(\eps \e_i)/\eps=1_{i=j}\lim_{\eps\to 0}a(\eps \e_i).
$$
To prove \eqref{NN8}, it suffices to show that
\begin{align}\label{LJ1}
0=\lim_{z\to 0}|a(z)^\alpha-\kappa(0)|=\lim_{r\to 0}\sup_{|\omega|=1}|(\psi(r,\omega)/r)^\alpha-\kappa(0)|.
\end{align}
From \eqref{NN3}, one sees that
\begin{align}\label{LJ2}
\phi(r,\omega)=\left[\alpha\int^R_r\frac{\kappa(t\omega)}{t^{1+\alpha}}\dif t+R^{-\alpha}\right]^{-1/\alpha},
\end{align}
which implies that
$$
\left(r/\phi(r,\omega)\right)^{\alpha}=\alpha r^\alpha\int^R_r\frac{\kappa(t\omega)}{t^{1+\alpha}}\dif t+r^\alpha R^{-\alpha}.
$$
Hence, by \eqref{NN6},
\begin{align*}
\lim_{r\to 0}\sup_{|\omega|=1}|\left(\psi(r,\omega)/r\right)^\alpha-\kappa(0)|&=\lim_{r\to 0}\sup_{|\omega|=1}|\left(r/\phi(r,\omega)\right)^{\alpha}-\kappa(0)|\\
&=\lim_{r\to 0}\sup_{|\omega|=1}\left(\alpha r^\alpha\int^R_r\frac{|\kappa(t\omega)-\kappa(0)|}{t^{1+\alpha}}\dif t\right)\\
&=\lim_{r\to 0}\sup_{|\omega|=1} \int^{1}_{(r/R)^\alpha}|\kappa(r\omega/s^{1/\alpha})-\kappa(0)|\dif s=0,
\end{align*}
where the last step is due to the dominated convergence theorem. Thus we get \eqref{LJ1}.\\
\\
(iii) By definition  \eqref{NN3} and the change of variable $s=t/|z|$, we get
$$
\int^R_{a(z)|z|}\frac{\kappa(tz/|z|)\dif t}{t^{1+\alpha}}=\int^R_{|z|}\frac{\dif t}{t^{1+\alpha}}\Rightarrow
\int^{R/|z|}_{a(z)}\frac{\kappa(s z)\dif s}{s^{1+\alpha}}=\frac{1}{\alpha}\left(1-\frac{|z|^\alpha}{R^\alpha}\right).
$$
Assume $\kappa\in C^1(B_R)$. Taking the gradient for both sides with respect to $z$, we obtain
$$
\nabla a(z)=\frac{a(z)^{1+\alpha}}{\kappa(a(z)z)}\left(\int^{R/|z|}_{a(z)}\frac{\nabla \kappa(sz)}{s^\alpha}\dif s+\frac{|z|^{\alpha-2}z}{R^\alpha}(1-\kappa(Rz/|z|))\right).
$$
Since $\kappa\in[\kappa_0^{-1},\kappa_0]$, by \eqref{NN9} and elementary calculations, we have
\begin{align}\label{NN10}
|\nabla a(z)|\leq 
\left\{
\begin{aligned}
&\kappa_0^{2+1/\alpha}\left(\|\nabla\kappa\|_\infty R^{1-\alpha}/(1-\alpha)+(1+\kappa_0)R^{-\alpha}\right)|z|^{\alpha-1},\ \ & \alpha\in(0,1),\\
&\kappa_0^3\Big(\|\nabla\kappa\|_\infty\log(R\kappa_0/|z|)+(1+\kappa_0)R^{-1}\Big),\ \ &\alpha=1,\\
&\kappa_0^3\|\nabla\kappa\|_\infty/(\alpha-1)+\kappa_0^{2+1/\alpha}(1+\kappa_0)R^{-1},\ \ &\alpha\in(1,2).
\end{aligned}
\right.
\end{align}
Finally, noticing that
$$
\p_i\Phi^j(z)=z_j\p_i a(z)+1_{i=j}a(z),\ \ \p_i\Phi^j(0)=\kappa(0)^{1/\alpha}=1_{i=j}a(0),
$$
we have
$$
|\p_i\Phi^j(z)-\p_i\Phi^j(0)|\leq |z_j|\cdot|\p_i a(z)|+1_{i=j}|a(z)-a(0)|,
$$
which together with \eqref{NN10} gives the desired estimate.\\
\\
(iv) Let $\Phi^{-1}$ be the inverse of $\Phi$. We have 
$$
(\nabla\Phi)\circ\Phi^{-1}\cdot\nabla\Phi^{-1}=\mI\Rightarrow \nabla\Phi=(\nabla\Phi^{-1})^{-1}\circ\Phi.
$$
Therefore,
$$
\nabla\Phi^{-1}\cdot(\nabla^2\Phi)\circ\Phi^{-1}\cdot\nabla\Phi^{-1}+(\nabla\Phi)\circ\Phi^{-1}\cdot\nabla^2\Phi^{-1}=0,
$$
and
$$
|\nabla^2\Phi|\leq|\nabla\Phi|^3\cdot|\nabla^2\Phi^{-1}\circ\Phi|.
$$
By induction method, there is a $C=C(j,d)>0$ such that
\begin{align}\label{LU1}
|\nabla^j\Phi|\leq C\sum_{\gamma\in\sA}|\nabla\Phi|^{1+\sum_{i=2}^j i\gamma_i}\prod_{i=2}^j\big|(\nabla^i\Phi^{-1})\circ\Phi\big|^{\gamma_i},
\end{align}
where 
$$
\sA:=\left\{\gamma=(\gamma_2,\cdots,\gamma_j): \sum_{i=2}^j(i-1)\gamma_i=j-1\right\}.
$$
By (ii), one sees that
\begin{align}\label{LU2}
\|\nabla\Phi\|_\infty\leq C_R.
\end{align}
On the other hand, by definition, it is easy to see that
$$
\Phi^{-1}(z)=\phi(|z|,z/|z|)z/|z|=:h(z) g(z),
$$
where $g(z):=z/|z|$ and by \eqref{LJ2},
$$
h(z):=\phi(|z|,z/|z|)=\left[\alpha\int^R_{|z|}\frac{\kappa(tz/|z|)}{t^{1+\alpha}}\dif t+R^{-\alpha}\right]^{-1/\alpha}.
$$
By elementary calculations, one finds that
$$
|\nabla^j h(z)|\leq C|z|^{1-j},\ \ |\nabla^j g(z)|\leq C|z|^{-j},
$$
and so,
\begin{align}\label{LU3}
|\nabla^j\Phi^{-1}(z)|\leq C|z|^{1-j}.
\end{align}
Substituting \eqref{LU3} and \eqref{LU2} into \eqref{LU1}, and by \eqref{NN9} and $\Phi(z)=a(z)\cdot z$, we obtain (iii).
\end{proof}
\bc\label{Cor72}
Given $R\in(0,\infty)$, $\kappa_0\in[1,\infty)$ and $d_0\in\mN$, let $\kappa(x,z):\mR^{d_0}\times B_R\to[\kappa_0^{-1},\kappa_0]$ be a measurable function. 
For any $\alpha\in(0,2)$, there is a map $\Phi(x,z): \mR^{d_0}\times B_R\to B_R$ such that for any nonnegative measurable function $f$,
$$
\int_{B_R}f\circ\Phi(x,z)\frac{\dif z}{|z|^{d+\alpha}}=\int_{B_R}f(z)\frac{\kappa(x,z)}{|z|^{d+\alpha}}\dif z.
$$
Moreover, $\Phi$ enjoys the following properties:
\begin{enumerate}[(i)]
\item $\Phi(x,0)=0$ and if $\kappa(x,-z)=\kappa(x,z)$, then $\Phi(x,-z)=-\Phi(x,z)$.
\item For $x\in\mR^{d_0}$, if $\kappa(x,\cdot)$ is continuous at $0$, then $\Phi(x,\cdot)$ is differentiable at point zero and
$$
\nabla_z \Phi(x,0)=\kappa(x,0)\mI.
$$
\item If $\kappa(x,\cdot)\in C^1(B_R)$ and $\|\nabla_z\kappa\|_\infty<\infty$, 
then there are $\beta\in(0,1)$ and $C>0$ such that for all $x\in\mR^{d_0}$ and $z\in B_R$,
$$
|\nabla_z \Phi(x,z)-\nabla_z \Phi(x,0)|\leq C|z|^\beta.
$$
\item If $\kappa\in C^\infty_b(\mR^{d_0}\times B_R)$, then for all $i,j\in\mN_0$, there is a $C_{ij}>0$ such that for all $x\in\mR^{d_0}$ and $z\in B_R$,
$$
|\nabla_x^i\nabla_z^j\Phi(x,z)|\leq C_{ij}|z|^{1-j},
$$
where $C_{ij}$ is a polynomial of $\|\nabla^m_x\nabla^n_z\kappa\|_\infty$, $m=1,\cdots$, $i, n=0,\cdots,j$.
\end{enumerate}
\ec
\begin{proof}
(i), (ii) and (iii) follow by (i), (ii) and (iii) of Lemma \ref{Le71}. As for (iv), it follows by similar calculations as in the proof of (iv) of Lemma \ref{Le71}.
\end{proof}

Now we can give the proof of Theorem \ref{Th71}.

\begin{proof}[Proof of Theorem \ref{Th71}]
Fix $\delta\in(0,1)$ being small. Define
$$
\sL_0 u(x,\v):={\rm p.v}\int_{|w|<\delta}(u(x,{\rm v}+w)-u(x,{\rm v}))\frac{\kappa(x,{\rm v},w)}{|w|^{d+\alpha}}\dif w,
$$
and
$$
\sK_0 u(x,\v):=\sL_0 u(x,\v)+{\rm v}\cdot\nabla_x u(x,{\rm v})+b(x,{\rm v})\cdot\nabla_{\rm v} u(x,{\rm v}).
$$
Then we can write
$$
\sK u(x,\v)=\sK_0u(x,\v)+\sL_1u(x,\v),
$$
where
$$
\sL_1 u(x,\v):=\int_{|w|\geq \delta}(u(x,{\rm v}+w)-u(x,{\rm v}))\frac{\kappa(x,{\rm v},w)}{|w|^{d+\alpha}}\dif w.
$$
For each $m$ and $p\geq 1$, by the chain rule, it is easy to see that
$$
\|\sL_1 u\|_{m,p}\leq C\|u\|_{m,p}.
$$
On the other hand, by Corollary \ref{Cor72}, there exists a function $g(x,\v,\cdot):B_\delta\to B_\delta$ so that
$$
\sL_0 u(x,\v)={\rm p.v}\int_{|w|<\delta}(u(x,{\rm v}+g(x,\v,w))-u(x,{\rm v}))\frac{\dif w}{|w|^{d+\alpha}},
$$
and operator $\sK_0$ satisfies {\bf (H$'_2$)}$+${\bf (H$^{\rm o}_g$)}$+${\bf (H$^{\rm uni}_{\bf or}$)}.
The desired result follows by Theorem \ref{Main}.
\end{proof}

\end{document}